\documentclass[12pt]{amsart}

\usepackage{latexsym, amssymb}

\newcommand{\be}{\begin{equation}}
\newcommand{\ee}{\end{equation}}
\newcommand{\beq}{\begin{eqnarray}}
\newcommand{\eeq}{\end{eqnarray}}

\newtheorem{thm}{Theorem}[section]
\newtheorem{prop}{Proposition}[section]
\newtheorem{cor}{Corollary}[section]
\newtheorem{lma}{Lemma}[section]
\newtheorem{df}{Definition}[section]
\newtheorem{rmk}{Remark}[section]

\def\II{\text{\rm II}}
\def\Pi{\displaystyle{\mathbb{II}}}

\def\p{\partial}

\def\R{\mathbb{R}}
\def\H{\mathbb{H}}

\def\p{\partial}
\def\lf{\left}
\def\ri{\right}
\def\e{\epsilon}
\def\ol{\overline}
\def\R{\Bbb R}
\def\wt{\widetilde}
\def\la{\langle}
\def\ra{\rangle}

\def\l{\lambda}
\def\Ric{\text{\rm Ric}}

\def\Dg{\Delta_g}

\begin{document}
\title{Einstein and conformally flat critical metrics of the volume functional}

\date{}
\renewcommand{\subjclassname}{\textup{2000} Mathematics Subject Classification}
\author{Pengzi Miao
and Luen-Fai Tam$^1$}
\thanks{$^1$Research partially supported by Hong Kong RGC General Research Fund  \#GRF 2160357}

\address{The School of Mathematical Sciences, Monash University,
Victoria, 3800, Australia.} \email{Pengzi.Miao@sci.monash.edu.au}

\address{The Institute of Mathematical Sciences and Department of
 Mathematics, The Chinese University of Hong Kong,
Shatin, Hong Kong, China.} \email{lftam@math.cuhk.edu.hk}

\begin{abstract}
Let $R$ be a constant. 
Let $\mathcal{M}^R_\gamma $ be the space of 
smooth metrics $g$ on a given compact manifold  
$  \Omega^n  $ ($n\ge 3$) with smooth boundary 
$\Sigma $ such that $ g $ has constant scalar curvature 
$ R $ and $g|_{\Sigma}$ is a fixed
metric $\gamma$ on $\Sigma$. Let $V(g)$ be the volume of
$g\in\mathcal{M}^R_\gamma$.    In this work, 
we classify all Einstein or
conformally flat metrics which are critical points of 
$V( \cdot )$ in $\mathcal{M}^R_\gamma$.
\end{abstract}

\maketitle
 \markboth{Pengzi Miao  and Luen-Fai Tam}
 { }

\section{Introduction}

In \cite{MiaoTam08}, the authors studied variational
properties of the volume functional, constraint to the
space of metrics of constant scalar curvature   with a
prescribed boundary metric, on a given compact manifold
with boundary. More precisely,   let $  \Omega^n  $ ($n\ge
3$) be a connected, compact $ n $-dimensional manifold with
smooth boundary $ \Sigma $ with   a fixed boundary metric
$\gamma$. Let $R$ be a constant.   Let $
\mathcal{M}^R_\gamma $ be the space of metrics on $ \Omega
$ which have constant scalar curvature $ R $ and have
induced metric on $ \Sigma $ given by $ \gamma $. It was
proved  in \cite{MiaoTam08} that if $g\in \mathcal{M}^R_\gamma$
is an element such that 
 the first  Dirichlet  eigenvalue of
$ ( n -1 ) \Dg + R $ on $ \Omega $  is positive, then  $
\mathcal{M}^R_\gamma $ has a manifold structure near $ g $.
Hence one can study variation of the volume functional near
$g$ in $\mathcal{M}^R_\gamma$. The authors \cite{MiaoTam08}
proved that: {\it   $ g $ is a critical point of the usual
volume functional $ V (\cdot) $ in $ \mathcal{M}^R_\gamma  $
if and only if there is a
function $ \l $ on $ \Omega $ such that $ \l = 0 $ at $
\Sigma $ and \be \label{critical-inte-s0}
 -(\Delta_g\lambda)g+\nabla^2_g\lambda-\lambda \Ric(g)  =g
 \ \ \mathrm{on} \ \Omega, 
\ee 
where $ \Dg$, $ \nabla^2_g $ are the Laplacian, Hessian
operator with respect to the metric $ g $ and $ \Ric (g) $
is the Ricci curvature of $ g $.}

The above result suggests the following definition:

\begin{df}
Given a compact manifold $ \Omega $ with smooth boundary,
we say a metric $ g $ on $ \Omega $ is  a {\bf critical metric}
if $ g $ satisfies \eqref{critical-inte-s0} for some function
$ \l $ that vanishes on the boundary of $ \Omega$. 
\end{df}

It was  shown in \cite{MiaoTam08} that
equation \eqref{critical-inte-s0} alone indeed
 implies that $ g $ has constant scalar curvature.
 Hence, a critical metric
 necessarily has constant scalar curvature.

A natural question is to characterize  critical metrics.
We have the following results from \cite{MiaoTam08}:
  
 \begin{enumerate}
\item[(i)]  {\em If $ \Omega $
 is a bounded domain with smooth
 boundary in  a simply connected space form
 $ \mathbb{R}^n$, $ \mathbb{H}^n $ or $ \mathbb{S}^n$,
 then the corresponding space form metric is  a critical metric
 on $ \Omega $ if and only if $ \Omega $ is a geodesic ball
 (if $ \Omega \subset \mathbb{S}^n$, one
 assumes $ V (\Omega) < \frac{1}{2} V( \mathbb{S}^n) $).
 \item[(ii)] If $ g $ is a critical metric with zero
 scalar curvature on a compact manifold  $ \Omega $
such that the boundary of $ (\Omega, g)$ is isometric to
 a geodesic sphere $ \Sigma_0$ in  $ \R^n$,
 then $ V (g) \geq V_0$, where $ V_0 $ is
 the Euclidean volume enclosed by
  $ \Sigma_0 $. Moreover, 
 $ V(g) = V_0 $ if and only if
 $ (\Omega, g) $ is isometric to a Euclidean geodesic ball.}

\end{enumerate}

These results  suggest that critical metrics with a prescribed 
boundary metric seem to be rather rigid. For instance, 
we want to know if there exist non-constant sectional curvature 
critical metrics on a compact manifold whose boundary 
is isometric to a standard around sphere. If yes, what can 
we say about the structure of such metrics?

In this paper, we study this rigidity question under
certain  additional assumptions: We assume
the  manifold is  Einstein or is conformally flat.
Since space forms are both
Einstein and conformally flat, these considerations
are natural steps following
results in \cite{MiaoTam08}.
Our study of conformally flat critical metrics
are also motivated by the work of Kobayashi and Obata
\cite{Kobayashi,KobayashiObata1981}.

The first result we obtain in this work is the following:

\begin{thm}\label{Einstein-s0-t1}
Let $ (\Omega, g )$ be a
connected, compact,  Einstein manifold with a smooth
boundary $\Sigma$.  Suppose the metric $ g $ is a critical metric.
Then $ (\Omega^n,g) $ is isometric to a
geodesic ball in a simply connected space form $
\mathbb{R}^n $, $ \mathbb{H}^n $ or $ \mathbb{S}^n$.

\end{thm}

To understand conformally flat critical metrics,
we first construct explicit examples of critical metrics
which are in the form of warped products.
It is interesting to note that those examples include the usual spatial Schwarzschild metrics and Ads-Schwarzschild metrics restricted to certain domains containing their horizon
 and bounded by two spherically symmetric spheres 
(see Corollary  \ref{Schwarzschild} and  \ref{Ads-Schwarzschild}).
Then we show that any conformally flat, non-Einstein, critical metric is
either one of the warped products we construct
or it is covered by such a metric.
More precisely, we have:

\begin{thm}\label{CF-s0-t2} Let $(\Omega^n,g)$ be a
connected, compact,  conformally flat
manifold with a smooth boundary $\Sigma$.
Suppose the metric $ g $ is a critical metric and
 the first Dirichlet eigenvalue
of $(n-1)\Dg+R$ is nonnegative,
where $ R $ is the scalar curvature of $ g $.
 \begin{enumerate}
\item[(i)]
If $\Sigma$ is disconnected, then $ \Sigma $ has
exactly two  connected components, and
$(\Omega,g)$ is isometric to  $ (I\times N, ds^2+r^2 h)$
where $I$ is a finite
interval in $\R^1$ containing the origin $ 0$,
$(N,h)$ is a closed manifold with
constant sectional curvature $\kappa_0$,
$r $ is a positive function on $ I $ satisfying
$ r^\prime (0) = 0 $ and
$$
  r^{\prime \prime} + \frac{ R}{n ( n -1) } r = a r^{1-n}
$$
for some constant $ a>0 $, and the constant $ \kappa_0
$ satisfies
$$   (r^\prime)^2 + \frac{R}{n(n-1)} r^2  +
\frac{ 2 a}{ n-2} r^{2-n} = \kappa_0  .
$$
\item[(ii)]
If $\Sigma$ is connected, then
$(\Omega, g)$ is either isometric to a
geodesic ball in a simply connected space form $
\mathbb{R}^n $, $ \mathbb{H}^n $, $ \mathbb{S}^n$,
or $(\Omega,g)$ is covered
by one of the above mentioned warped products in  (i)  with
a covering group $\mathbb{Z}_2$.
\end{enumerate}
\end{thm}

It follows from Theorem \ref{CF-s0-t2} that
if $ g $ is a conformally flat critical metric on a simply connected
manifold $ \Omega $ such that the boundary of 
$(\Omega, g)$ is isometric to a  standard round sphere,
then $(\Omega,g)$ is isometric to a geodesic
ball in $\mathbb{R}^n $, $ \mathbb{H}^n $ or $ \mathbb{S}^n$.

The organization of the paper is as follows.
In section 2, we consider critical metrics which
are Einstein. We prove that compact manifolds with  critical Einstein metrics are geodesic balls in simply connected space forms.
In section 3, we construct  critical metrics which
can be written as a warped product or the quotient of a
warped product.  In particular, we
obtain non-Einstein  critical metrics whose boundary
is a standard round sphere and   examples
of critical metrics whose boundary is disconnected.
In section 4, we classify all conformally flat critical metrics.
We prove that they are exactly the metrics constructed in section 2.
For completeness and easy reference, we include an appendix on  estimates
of graphical representation  of hypersurfaces with bounded
second fundamental form, which is needed in Section 4.
  All manifolds considered in this paper are assumed to be connected with dimension  $ n\geq 3 $.

\section{critical Einstein metrics}

Let $(M,g)$ be an Einstein manifold with or without boundary. 
We normalize $ g $ so that 
$ \Ric(g) = ( n -1) \kappa g $, 
where $ \kappa = 0$, $ 1$, or $ -1$.
Suppose there is a non-constant  function $ \l $ on $ M $ satisfying
\begin{equation}\label{critical-inte-s1}
  -(\Delta_g\lambda)g+\nabla^2_g\lambda-\lambda \Ric(g)  =g .
\end{equation} 
We will prove in Theorem \ref{Einstein-t1} 
that, if $ M $ is connected, compact with nonempty  
boundary on which $ \l $ is zero, then $ (M, g)$
is isometric to a geodesic ball in 
$ \mathbb{R}^n $, $ \mathbb{H}^n$ or $ \mathbb{S}^n$. 
In Theorem \ref{Einstein-t2}, we will  also classify  those $(M, g)$ 
that are  complete without boundary. 

We note that all geodesics in this section are assumed
to be parametrized by arc-length.  

\begin{lma}\label{Einstein-l1} 
Let $ (M, g) $ and $ \l $ be given as above. Suppose
 there exists  $ p \in M $ such that $ \nabla \l (p) = 0 $.
Then the followings are true:

\begin{enumerate}
\item[(i)] Along a geodesic $\alpha(s)$ emanating from $p$, we have:
    \begin{itemize}

\item [(a)] if  $ \kappa = 0 $, then
$$ \lambda(\alpha(s)) = -\frac{1}{2 ( n -1 )} s^2 + \l(p); $$
\item[(b)] $ \kappa = 1$, then
$$\lambda(\alpha(s))=\lf(\l(p)+\frac1{n-1}\ri) \cos  s - \frac{1}{ ( n-1) }.  $$
\item[(c)] $ \kappa =-1 $, then
$$ \lambda(\alpha(s))= \lf(\l(p)-\frac1{n-1}\ri) \cosh s + \frac{1}{ ( n-1) } . $$

\end{itemize}

   \item[(ii)]   Suppose  $  q \in M $ such that 
   there exists a minimizing geodesic $ \alpha (s) $ 
 connecting $ p $ to $ q $.
 If $ \beta (s) $ is another geodesic connecting 
 $ p $ to $ q $ and $ \beta (s) $  has length no
  greater than $ \pi $ if $ \kappa = 1$, 
  then $ \beta (s) $ is also minimizing.
  \end{enumerate}
\end{lma}

\begin{proof}  As $ \Ric(g) = ( n -1 ) \kappa g $, 
\eqref{critical-inte-s1} is equivalent to 
\be \label{equi-eq-s1}
\nabla^2_g \l = \lf( - \kappa \l - \frac{1}{n-1} \ri) g .
\ee
Hence, $ \l $ satisfies 
\be 
\frac{d^2}{ds^2}\lambda(\alpha(s))=-\kappa \l ( \alpha (s) ) -\frac1{n-1},
\ee
along $ \alpha(s) $. From this and the fact  $ \nabla \l (p) = 0 $, 
(i) of 
the Lemma follows. 

To prove (ii), let $ r $ and $ l $ be the length of $ \alpha(s) $ 
and $ \beta(s)$. By (i) and the fact $ \alpha (r) = q = \beta(l)$, 
we have:
$$
-\frac{1}{2(n-1)} r^2 + \l(p)= -\frac{1}{2(n-1)} l^2 + l(p)
$$
if $\kappa=0$;
$$
\lf(\l(p)+\frac1{n-1}\ri) \cos  r - \frac{1}{ ( n-1) } =\lf(\l(p)+\frac1{n-1}\ri) \cos  l - \frac{1}{ ( n-1) }.
$$
if $\kappa=1$; and
$$
\lf(\l(p)-\frac1{n-1}\ri) \cosh r + \frac{1}{ ( n-1) }= \lf(\l(p)-\frac1{n-1}\ri) \cosh l + \frac{1}{ ( n-1) }
$$
if $\kappa=-1$. Since $ \l $ is not identically a constant,
we have $ \l (p) + \frac{1}{n-1} \neq 0 $ if $ \kappa = 1 $
and $ \l (p) - \frac{1}{n-1} \neq 0 $ if $ \kappa = - 1 $. 
 In case $\kappa=0$ or $-1$, it is then evident that 
 $r=l$. In case $\kappa=1$, we have 
$ \Ric(g) = ( n -1) g $, which implies  $ r \leq \pi $
as $ \alpha (s) $ is minimizing.
Since $ l \leq \pi$ by assumption, we have $ r = l $. 
This shows that $ \beta (s) $ is also minimizing. 
 
\end{proof}

\begin{lma}\label{Einstein-l2}
Let $(M, g)$ and $ \l $ be given as above. 
 Suppose $\Sigma \subset M $ is a connected, embedded  hypersurface 
 on which  $ \l $ equals a constant. 
 Suppose  $ \nabla \l  $ never vanishes on $ \Sigma $ 
 and let $ \nu = \nabla \l/|\nabla\l|$. Then $ | \nabla \l | $ is 
constant on $ \Sigma $ and
 the second fundamental
form $A( X, Y)$  of $ \Sigma $ 
with respect to $ \nu $
satisfies
       \begin{equation}\label{Einstein-2ndfund-e1}
       A (X,Y)= |\nabla\l|^{-1}\lf(-\kappa\l-\frac1{n-1}\ri)g(X,Y) ,
       \end{equation}
      where  $X, Y$ are any tangent vectors to $\Sigma.$

\end{lma}
\begin{proof}
Using  the fact that $ \l $ equals a constant on $ \Sigma $, we have
\be \label{nablal}
\begin{split}
\frac{1}{2} X ( | \nabla \l |^2 ) & = \la \nabla_X ( \nabla \l ), \nabla \l \ra \\
& = | \nabla \l | \la  \nabla_X ( \nabla \l ) , \nu \ra \\
& = | \nabla \l | \nabla^2_g ( \l ) ( X, \nu)
\end{split}
\ee
and
\begin{equation} \label{IIofSigma}
    \begin{split}
            A (X,Y) &=  \la \nabla_X \nu, Y \ra \\
            & = | \nabla \l |^{-1}  \la \nabla_X ( \nabla \l ),  Y \ra \\
              &=|\nabla\l|^{-1}\nabla^2_g(\l)(X,Y).
          \end{split}
\end{equation}
From \eqref{equi-eq-s1}, \eqref{nablal}
and \eqref{IIofSigma}, we conclude that $ X (  | \nabla \l |^2 )= 0 $ 
and \eqref{Einstein-2ndfund-e1} holds. 
\end{proof}

\begin{thm}\label{Einstein-t1} Suppose $ (\Omega, g )$ is a 
connected, compact,  Einstein manifold with a smooth boundary $\Sigma$.  Suppose there is a function $ \lambda $ on
$ {\Omega} $ such that $ \l = 0 $ on $ \Sigma $ and
\begin{equation}\label{critical-inte2-s1}
  -(\Delta_g\lambda)g+\nabla^2_g\lambda-\lambda \Ric(g)  =g 
\end{equation} 
in $ \Omega $.
Then $ (\Omega^n,g) $ is isometric to a geodesic ball in
a simply connected space form $ \mathbb{R}^n $, $ \mathbb{H}^n $
or $ \mathbb{S}^n$.

\end{thm}

\begin{proof}
We normalize $ g $ such that $ \Ric( g) = ( n -1 ) \kappa g $, 
where $ \kappa = 0$, $ 1$, or $ -1$.
Since $\l=0$ on $\Sigma$ and $\l$ is not identical zero, 
there exists an interior point $p\in {\Omega}$ such that $\nabla\l(p)=0$.
Let $ r_0 = $ dist$ (p, \Sigma)$, the distance from $ p $ to $\Sigma $. 
Consider the geodesic ball $ B_{r_0}(p) \subset {\Omega}$
centered at $ p $ with radius $ r_0 $. 
Then $ \p B_{r_0} (p) \cap \Sigma \neq \emptyset $.
By Lemma \ref{Einstein-l1}, we 
have $ \l = 0 $ on $ \p B_{r_0} (p)$. 

Suppose $ \kappa = 0 $. Then \eqref{equi-eq-s1} 
implies $  \Dg \l  < 0$. By the maximum principle,
we must have    $ \p B_{r_0}(p) \subset \Sigma $.
As $ \Omega $ is connected, we have
$ B_{r_0}(p) = {\Omega} $.
Furthermore, the fact $ r_0 = $ dist$( p, \Sigma)$ implies 
  every  geodesic $ \alpha(s) $ emanating from $ p $ 
 is minimizing on $ [0, r_0] $  and every $ q \in \Sigma $ can be
  connected to  $ p $ by a unique minimizing geodesic with
 length $ r_0 $. It follows that 
  the exponential map at $ p $ is a diffeomorphism 
 onto $ B_{r_0} (p) = {\Omega} $.
For each $ s \in (0, r_0] $,  let  $ \Sigma_s  $ be the embedded geodesic sphere centered 
 at $ p $ of radius $ s $. By Lemma \ref{Einstein-l1}, 
 $ \l = - \frac{1}{2(n-1)} s^2 + \l (p) $ on $  \Sigma_s $. 
 In particular, $ \nabla \l $ does not vanish on $ \Sigma_s $. 
 Let $ H_s $ be the mean curvature of $ \Sigma_s $ 
 w.r.t the outward unit normal. 
 By Lemma \ref{Einstein-l2}, we have 
 $ H_s = \frac{n-1}{s} $. Let $ A(s) $ be the areas of $ \Sigma_s $.
 Then $ \frac{d}{ds} A(s) = \frac{n-1}{s} A(s) $. From this it follows 
 that the volume of $ (\Omega, g) $ 
 agrees with  the volume of a geodesic ball of radius $ r_0 $ in 
 $ \mathbb{R}^n $. Since $ \Ric (g) = 0 $,  by the Bishop volume comparison theorem \cite{BishopCrittenden}, we conclude that
 $ (\Omega, g) $  is isometric to a geodesic ball in $\R^n$.

Suppose  $\kappa=-1$, then  \eqref{equi-eq-s1} implies 
$\Dg\l-n\l<0$. The maximum principle can still be applied 
to show $ \p B_{r_0}(p) \subset \Sigma $. Hence
we can prove similarly  that  $ (\Omega, g)$ is 
isometric to a geodesic ball  in $ \mathbb{H}^n $.

Finally, suppose $\kappa=1$. Since $ \Ric(g) = ( n -1) g $,
we have $ r_0 \leq \pi $. In particular, the function
$ f(s) = ( \l(p) + \frac{1}{n-1} ) \cos{s} - \frac{1}{n-1} $
has nowhere vanishing derivative on $(0, r_0]$. 
If  $ \l $ never vanishes  in the interior of ${\Omega}$,
 we can proceed as before to show that 
$ (\Omega, g)$ is isometric to a geodesic ball in $ \mathbb{S}^n $. 
In general, let $ \Lambda_0 $ be the set of interior
points where $  \l $ vanishes. 
Suppose $ \nabla \l ( q ) = 0 $ for some $ q \in \Lambda_0$.
Let $ d  = $ dist$( q, \Sigma)$ and let $ \beta(s) $ be 
a geodesic  such that $ \beta (0) = q $
and $ \beta (d ) \in \Sigma $. By Lemma \ref{Einstein-l1}
and the fact $ \l  ( q ) = 0 $,
we have
$ \l ( \beta (s) ) =  \frac{1}{n-1} \cos{s} - \frac{1}{n-1} $.
At $ s = d$, we have $ \l ( \beta ( d ) ) = 0 $, hence 
$ \cos{ d } = 1 $. On the other hand, the fact 
  $ \Ric(g) = ( n-1) g $ implies $ d  \leq \pi $, which is  a contradiction. Therefore, $ \nabla \l $  never vanishes at points in $ \Lambda_0$.
In particular,  $ \Lambda_0 $ is an embedded 
hypersurface in $ {\Omega} $.

 Let $\Sigma_1$ be a connected component of $\Sigma$.
 At $ \Sigma_1 $, we have   $ \nabla^2_g \l = - \frac{1}{n-1} g $ by \eqref{equi-eq-s1}.  As mentioned in \cite{MiaoTam08},  this implies that the  mean curvature $H$ of 
$ \Sigma_1 $ (w.r.t the outward unit normal $ \nu $)
 satisfies $ H\frac{\p \l}{\p \nu}=-1.$  In particular, $\frac{\p \l}{\p \nu}$ never vanishes on $\Sigma_1$. 
 Suppose $ \frac{\p \l}{\p \nu} < 0 $ on $ \Sigma_1 $. 
 Since $\l=0$ on $\Sigma_1$,  there exists a connected open set $ U_1 $  in $ \Omega $ containing $ \Sigma_1 $ such that   $ \l > 0 $  on $ U_1 \setminus \Sigma_1 $. 
  Consider the open set 
  $ \Omega^+ = \{ q \in {\Omega} \ | \ \l ( q) > 0 \} $.  
  Let $ \Omega^+_1 $ be the connected component of 
  $ \Omega^+ $ containing $ U_1   \setminus \Sigma_1 $. 
  Let $ \ol{\Omega}^+_1$ be the closure of $ \Omega^+_1 $
  in $ \Omega $. Then $ \ol{\Omega}^+_1 $ is a compact manifold
  with smooth nonempty boundary $ \p  \ol{\Omega}^+_1 $,
    moreover $ \l > 0 $ in $ \Omega^+ $ and 
  $ \l = 0 $ on $ \p  \ol{\Omega}^+_1$. 
  Replacing $ \Omega $ by $ \ol{\Omega}^+_1 $, 
  we can prove as before that $( \ol{\Omega}^+_1, g) $ is 
  isometric to a geodesic ball in $ \mathbb{S}^n $. 
  In particular, $ \p \ol{\Omega}^+_1$ is connected. 
  Since $ \Sigma_1 \subset  \p \ol{\Omega}^+_1$, we 
  must have $ \Sigma_1 = \p \ol{\Omega}^+_1$. 
  Consequently, $ \ol{\Omega}^+_1 $ is  an open set 
  in $ \Omega$. Since $ \Omega $ is connected, we conclude 
  that $ \Omega =  \ol{\Omega}^+_1 $ and
 $ (\Omega, g)$ is   isometric to a geodesic ball 
 in $ \mathbb{S}^n $. 
  The case  $ \frac{\p \l}{\p \nu} > 0 $ on $ \Sigma_1 $
  can be proved similarly by considering 
  $ \Omega^- = \{ q \in {\Omega} \ | \ \l ( q) < 0 \} $.   
\end{proof}

Next we  consider complete Einstein manifolds 
$(M, g)$ that admit a non-constant solution $ \l $ 
to  \eqref{critical-inte-s1}.

\begin{thm}\label{Einstein-t2} Let $(M^n,g)$ be a connected, complete manifold without boundary. Suppose $g$ is Einstein with $\Ric(g)=(n-1)\kappa g$ where $\kappa=0, 1$  or $-1$. Suppose there exists a non-constant solution $ \l $ to the equation
 \be \label{critical-inte-ein}
 -(\Delta_g\lambda)g+\nabla^2_g\lambda-\lambda \Ric(g)  =g.
\ee
\begin{enumerate}
  \item [(i)] If $\kappa=1$, then $(M^n,g)$ is isometric to 
  $ \mathbb{S}^n$.
  \item [(ii)] If $\kappa=0$, then $(M^n,g)$ is isometric to 
  $ \mathbb{R}^n$. 
  \item [(iii)] If $\kappa=-1$, then $(M^n, g)$ is isometric 
  to $ \H^n $  provided 
  $\nabla\l(p)=0$ for some $p$. 
    If $\nabla \l \neq0$ everywhere, then $(M, g)$ is isometric 
  to $ (\R^1 \times \Sigma,d s^2+\cosh ^2 s g_0)$ and  $\l $ is 
  given by $A\sinh s+\frac1{n-1}$ for some constant $A>0$. Here $(\Sigma, g_0)$ is a
  complete  Einstein manifold satisfying
   $\Ric(g_0)=-(n-2)g_0$. In particular, $(M^n, g)$ has 
   constant sectional curvature $ -1 $ if $ n \leq 4$. 
\end{enumerate}

\end{thm}
\begin{proof} (i) If $ \kappa = 1 $, then $M$ is compact with diameter $ d \leq \pi$.   Choose  $p\in M$ such that  $\nabla\l(p)=0$. 
Let $ \alpha(s) $ be a geodesic defined on $[0, \infty)$ 
with $ \alpha(0) = p $. 
By  (i) in Lemma \ref{Einstein-l1}, $ \l ( \alpha ( \pi ) )  \neq \l (p) $,
hence $ \alpha(\pi) \neq p $.  By (ii) in Lemma \ref{Einstein-l1},
$ \alpha(s) $ is  minimizing  on $[0, \pi] $. Hence,  $ d \geq \pi $. 
Therefore $ (M, g) $ is isometric to $ \mathbb{S}^n $
by the maximal diameter theorem \cite{Cheng75}.

(ii) Suppose $ \kappa = 0 $, we show that $ \l $ must have an absolute  maximum.  Let $ q \in M $ be any given point.  
The exponential 
map $ exp_q( \cdot ): T_q M \rightarrow M $ is surjective, 
where $ T_q M $ is the tangent space of $ M$ at $ q $.
Define $ \tilde{ \l } = \l \circ exp_q $. Let $ S_q $ be the unit sphere
in $ T_q M $.  For any $ v \in S_q $ and any $ s \geq 0 $, 
\eqref{equi-eq-s1} implies
\be \label{comeq-s1}
 \frac{ d^2} {d s^2} \tilde{ \l } ( s v ) = - \frac{1}{n-1}  . 
 \ee
Since
$ \tilde{ \l } ( 0 ) = \l (q) $ and $ 
\frac{d}{ds}  \tilde{ \l } ( s v ) (0) = \la \nabla \l (q) , v \ra 
$,  \eqref{comeq-s1} implies 
\be
 \tilde{ \l } ( s v )  = - \frac{1}{2(n-1)} s^2 + 
 \la \nabla \l (q) , v \ra s + \l (q) . 
\ee
Since $ | \la \nabla \l (q) , v \ra | \leq | \nabla \l (q) | $, 
we have 
$
\lim_{s \rightarrow \infty}  \tilde{ \l } ( s v )  = - \infty
$
uniformly with respect to $ v \in S_q $. 
In particular,  $ \tilde{ \l } $ has an absolute maximum.
Therefore,  $ \l $ has an absolute maximum. Consequently, 
there exists $ p \in M $ such that $ \nabla \l ( p ) = 0 $.
 By (ii) in Lemma \ref{Einstein-l1},  
  the injectivity radius of $(M, g)$ at $ p $ is $ \infty$.
Hence,  we can proceed as in the proof of Theorem \ref{Einstein-t1} to conclude that $(M^n,g)$ is isometric to $\R^n$.

(iii) Suppose $ \kappa = - 1$. If $\nabla\l=0$ somewhere, 
we can proceed as in the proof of Theorem \ref{Einstein-t1} to conclude that $(M^n,g)$ is isometric to the hyperbolic space $\H^n$.
In what follows, we assume that  $\nabla\l$ is never zero. 
For  $a\in \R$, let $\l_a$ be   the level set $\{\l=a\}$. Then $\l_a$  is a smooth hypersurface whenever it is nonempty. By Lemma \ref{Einstein-l2}, $|\nabla\l|$ is constant on each connected 
component of $\l_a$.

Choose $a$ such that $\l_a$ is nonempty. Let $\Sigma$ be a connected component of $ \l_a $ and let $b > 0 $ 
be the constant that equals $ | \nabla \l | $ on $ \Sigma$.  
Let $p\in \Sigma$  be any chosen point.
Let $\gamma(s)$ the geodesic defined on $(-\infty, \infty)$
such that $ \gamma (0) = p $ and 
$\gamma'(0) = b^{-1} \nabla \l (p) $.  Then $f(s)=\l(\gamma(s))$ satisfies $ f(0) = a $, $ f' (0) = b $ and
\begin{equation}\label{k=0-e1}
 \frac{d^2f}{ds^2} - f  = - \frac{1}{n-1}.
\end{equation}
Suppose $f'(s)=0$ for some $s>0$. 
Let $s_1>0$ be the smallest $s>0$ such that $f'(s)=0$.
For any $ 0 < s  < s_1$, consider the level set $ \l_{f(s)} $.
Let  $ s' = $ dist$(p, \l_{f(s)} )$, then
$ s' \leq s $. Let $\alpha (\cdot)$ be a minimizing geodesic such that 
$ \alpha (0) = p $ 
 and $ \alpha (s') \in \l_{f(s)} $. Let  $F(s)=\l(\alpha(s))$. 
 Then $F$ also satisfies \eqref{k=0-e1} with $F(0)=a$ and $F'(0)\le b$. If $F'(0)<b$, by \eqref{k=0-e1} we have
  $f(s')>F(s')=f(s)$. On the other hand, the facts  
 $ s' \leq s $ and $f$ is strictly increasing on $[0,s]$ imply 
 $ f (s') \leq f(s) $, hence a contradiction. 
 Therefore, $F'(0)=b$. In this case, we have 
 $ \alpha'(0) = \gamma'(0) $, hence $ \alpha ( t)  = \gamma (t) $ 
 for all $ t \in [0, s']$. Since $ \l ( \alpha (s') ) = \l ( \gamma (s) ) $,  
 we conclude $s'=s$.
 Consequently,  $ s = $ dist$(p, \l_{f(s)})$ and 
 $\gamma'(s)\perp \l_{f(s)}$
 at $ \gamma (s) $.  Since $ s \in (0, s_1) $ is arbitrary, 
we have $\gamma'(s_1) \perp \l_{ f(s_1) }$
 at $ \gamma (s_1) $. In particular, $ \gamma'(s_1) $ and
 $ \nabla \l ( \gamma(s_1) )$ are parallel, hence
 $ f' (s_1) = \la  \gamma'(s_1), \nabla \l ( \gamma(s_1) ) \ra \neq 0 $.
 This contradicts the assumption $ f'(s_1) = 0 $. Therefore, 
 $ f'(s) \neq 0 $ for all $ s > 0 $. Similarly, we can 
 prove that $f'(s)\neq0$ for $s<0$.

Now we have $ f'(s) > 0 $ for all $ s $. Moreover, by the above proof,
we have  $ \gamma'(s)\perp \l_{f(s)}$
 at $ \gamma (s) $ for all $ s $. Hence,
 \be
 \nabla ( \l (s) ) = \phi (s) \gamma' (s) 
 \ee
 for some smooth positive function $ \phi (s) $ defined on
 $(-\infty, \infty)$. Therefore, after reparametrization, $\gamma$ is an integral curve of the vector field $\nabla\l$. In particular, two different $\gamma$ will not intersect. Since any point in $M$ 
 lies on  a geodesic that  is perpendicular to $ \Sigma $, 
 we conclude that  $(M, g)$ is isometric to 
 $ (\R^1 \times \Sigma, d s^2 + g_s ) $, where 
 $ \{ s \} \times \Sigma $ is the level set of dist($\cdot, \Sigma$)
 and $ g_s $ is the induced metric on $ \{ s \} \times \Sigma  $. 
 Moreover, by \eqref{k=0-e1} and the fact
 $ \l $ and $ | \nabla \l | $ are constants on $ \Sigma  $, we know
 $ \l $ depends only on $ s $ and $ \l = \l (s) $ is given by
\begin{equation}\label{k<0-e1}
    \l( s )=A\sinh s+B\cosh s +\frac1{n-1}
\end{equation}
for some constants $ A $ and $ B$. 
Since $|\nabla\l|=|\l'|$, which is never zero, by reversing 
$ \frac{\p}{\p s} $, we may assume that $\l' (s) >0$ for all $ s $. 
Let  $A_s$ be the second fundamental form of 
$\{s \}\times\Sigma$ w.r.t $ \frac{\p}{\p s}$.
By  Lemma \ref{Einstein-l2} and  \eqref{k<0-e1}, we have
\begin{equation}\label{k<0-e2}
\frac{\p }{\p s}g_s = 2 A_s = 2 |\nabla\l|^{-1}\lf(\l-\frac1{n-1}\ri)g_s
= 2 \frac{\l''}{\l'}g_s .
\end{equation}
Therefore, we conclude $g_s =\phi^2(s)g_0$, where 
 $$\phi(s)=\frac{\l'(s )}{\l'(0)}=A^{-1}\lf(A\cosh s+B\sinh s \ri) . $$ Since $\l'>0$, we  have $A>0$ and  $A\ge |B|$. If $A=|B|$, then $\phi(s)=e^{s}$ or $e^{-s}$ and the metric $ g $ is not complete. Hence, $A>|B|$.
 Therefore, $\l=\frac1{n-1}$ somewhere. By translating $ s$,
 we may assume $ \l (0) = \frac1{n-1}$. Then 
  $\l (s) =A\sinh s + \frac{1}{n-1} $, $\phi( s)=\cosh s$, and 
  \be \label{warped-s1}
   g = d s^2 +  \cosh^2 s  g_0 .
   \ee
Using the fact  $ \Ric (g)  =  - ( n -1) g $ and  \eqref{eq-Ric2-s2} in Lemma \ref{warpedproductmetric-l1} in the next section, we have 
$ \Ric(g_0) = - ( n - 2) g_0. 
$
When $n=  4 $, this implies $ g_0$ has constant sectional 
curvature $ -1$, hence $ g $ has constant 
sectional curvature $ -1$ by \eqref{warped-s1}. 
\end{proof}

Let $(\Sigma,g_0)$ be any  complete Einstein manifold 
with negative scalar curvature which is not a space form. 
Suppose $ \Ric(g_0) = - ( n -1) g_0$. 
Consider the warped product 
$ (M, g) = (\R^1 \times \Sigma, d s^2 + \cosh^2 s g_0) $.
Define $ \l = A\sinh s+\frac1{n-1}$ on $ M $, where $A>0$
is  a constant. It is easy to  verify  that 
$ \l $ is a  solution to \eqref{critical-inte-ein}. 
In this case, $(M, g)$ is complete,  Einstein, but 
is not a space form.

\section{Warped-product critical metrics}
In this section, we first seek a general procedure
to construct warped-product metrics $ g $ which  
satisfy 
\be \label{critical-inte-s2}
 -(\Delta_g\lambda)g+\nabla^2_g\lambda-\lambda \Ric(g)  =g
\ee
for some function $ \l $. Then we construct examples of
critical metrics with disconnected 
boundary and non-Einstein critical metrics whose boundary is a standard round sphere.
The first part of our discussion is  motivated by the  work of Kobayashi in \cite{Kobayashi}. 

Let $ (N, h)$ be a Riemannian manifold of dimension
$ n-1$. Let $ I \subset \mathbb{R}^1 $ be an open interval
and $ d s^2 $ be the standard metric on $ I $.
Let $ r  $ be a smooth positive function on  $I $.
 Consider the  warped-product metric
 $$ g = d s^2 + r^2  h $$
 on $M = I \times N $.

\begin{lma}\label{warpedproductmetric-l1}
\begin{enumerate}
 \item[(i)]
 The Ricci curvature of $g$ is given by
\be \label{eq-Ric1-s2}
\begin{split}
 \Ric({g} )(\p_s, \p_s)  = & - ( n -1 ) \frac{ r^{\prime \prime} }{r} ,
\end{split}
\ee
\be \label{eq-Ric2-s2}
 \Ric(g) |_{ T N} = \Ric(h) - \lf[ ( n -2) \left(\frac{ r^\prime}{r} \right)^2
+ \frac{ r^{\prime \prime} }{r} \ri] g |_{TN} ,
\ee
\be \label{eq-Ric3-s2}
\Ric( \p_s, X) = 0 , \ \forall \ X \in TN ,
\ee
where $`` \ '  \ " $ denotes the derivative taken with respect to $ s \in I$,  $ \Ric(h)$ is the Ricci curvature of $ h $ and $ T N$
 denotes the tangent space to $ N$. Consequently,
 \be \label{eq-rRh-s2}
R (g) = - 2 ( n -1 ) \lf( \frac{ r^{\prime \prime} }{r} \ri)
+ \frac{R(h)}{r^2} - ( n-1) ( n -2) \left(\frac{ r^\prime}{r} \right)^2,
\ee
where $ R(g)$, $R(h) $ are  the scalar curvature of $ g$, $h$
respectively.

\item[(ii)] Suppose $\l$ is a smooth function on  $ M$ depending only on $s$, then

\be \label{eq-Hessl-s2}
\nabla^2_g \l (\p_s, \p_s) = \l^{\prime \prime}, \
\nabla^2_g \l |_{ TN} = \lf( \frac{ r^\prime }{ r} \right) \l^\prime g |_{TN} ,
\  \nabla^2_g \l (\p_s, X) = 0
\ee
where $ X \in TN$.
 \end{enumerate}
 \end{lma}
 \begin{proof} (i) is standard, see \cite{BishopONeill69}. Direct computations give (ii).
 \end{proof}

To proceed, we note that \eqref{critical-inte-s2}
implies 
\begin{equation}\label{critical-e2-s2}
\Dg\l=-\frac1{n-1}\lf( R\l+n\ri),
\end{equation}
Hence,  \eqref{critical-inte-s2} is equivalent to
\begin{equation}\label{critical-e3-s2}
\nabla^2_g\l=\l \Ric-\frac{R\l +1}{n-1}g. 
\end{equation}

  \begin{prop} \label{prop-1-s2}
For any  constant $ R $, the metric $ g $ has constant
scalar curvature $ R $ and satisfies  \eqref{critical-inte-s2}
for a smooth function $ \l $ depending
 only on $ s \in I$,  if and only if the following holds:
\begin{enumerate}

 \item[(i)] $(N, h)$ is an Einstein manifold with $ \Ric(h) = ( n-2) \kappa_0 h $,   the function  $ r $ satisfies
\be \label{eq-rRa-s2}
  r^{\prime \prime} + \frac{ R}{n ( n -1) } r = a r^{1-n}
\ee
for some constant $ a $, and the constant $ \kappa_0 $ satisfies
\be \label{eq-rRa-s2-2}
(r^\prime)^2 + \frac{R}{n(n-1)} r^2  + \frac{ 2 a}{ n-2} r^{2-n} =
\kappa_0.
\ee

\item[(ii)] The function  $ \l $ satisfies

\be\label{lambda-e1} r^\prime \l^\prime -  r^{\prime \prime} \l = -
\frac{1}{n-1} {r}  .
\ee
\end{enumerate}
\end{prop}

\begin{proof}
Suppose $g$ has constant scalar curvature $ R $ and there is a smooth function $\l = \l(s)$ satisfying \eqref{critical-inte-s2}. Since $\l$ can not be identically zero, there exists  $s_0\in I$  such that $\l(s_0)\neq0$. At $s_0$, by  Lemma \ref{warpedproductmetric-l1} and 
\eqref{critical-e3-s2}, we have

 \begin{equation}\label{warpedproductmetric-e1}
 \begin{split}
 \Ric(h)&=\Ric(g) |_{ T N} + \lf[ ( n -2) \left(\frac{ r^\prime}{r} \right)^2
+ \frac{ r^{\prime \prime} }{r} \ri] g |_{TN}\\
&=\frac1{\l}\lf(\nabla^2_g\l+\frac{R\l +1}{n-1}g\ri)|_{TN}+ \lf[ ( n -2) \left(\frac{ r^\prime}{r} \right)^2
+ \frac{ r^{\prime \prime} }{r} \ri] g |_{TN}\\
&=\lf[\frac1{\l}\lf( \frac {r' \l' }r +\frac{R\l +1}{n-1} \ri) +   ( n -2) \left(\frac{ r^\prime}{r} \right)^2
+ \frac{ r^{\prime \prime} }{r} \ri] g |_{TN} .
\end{split}
\end{equation}
Since $R$ is a constant and $r$ and $\l$ depend only on $s$,  \eqref{warpedproductmetric-e1}
implies  that $(N,h)$ is Einstein. Suppose $\Ric(h)=(n-2)\kappa_0 h$
where $ \kappa_0$ is a constant.

Evaluating both sides of \eqref{critical-inte-s2}  at $\p_s$, using Lemma \ref{warpedproductmetric-l1} and the fact that
$$
\nabla^2\l(\p_s,\p_s)-\Dg\l=-(n-1)\frac{ r^\prime}{r} \l'
$$
we have
$$
-(n-1)\frac{ r^\prime}{r} \l' +   ( n -1 ) \frac{ r^{\prime \prime} }{r}  \l
= 1,
$$
which proves (ii).

Differentiating \eqref{lambda-e1}, using \eqref{critical-e2-s2},
\eqref{lambda-e1}  and the fact that
$$
\Dg\l=\l''+(n-1)\frac{r'\l'}r ,
$$
 we have
\begin{equation*}
\begin{split}
       -\frac{r'}{n-1}& =r'\l''-r'''\l  \\
         & =\lf(\Dg\l-(n-1)\frac{r'\l'}r\ri)r' -r'''\l \\
         &= \lf(-\frac{R\l+n}{n-1}-(n-1)\frac{r''\l}{r}+1\ri)r'-r'''\l .
     \end{split}
\end{equation*}
Hence
\begin{equation} \label{rtriplep-e1}
 \lf[  r^{\prime \prime \prime} + ( n - 1)   \frac{ r^\prime r^{\prime \prime} }{r}  +
    \frac{R}{n-1} r^\prime  \ri] \l =  0 .
\end{equation}
By \eqref{lambda-e1}, if $\l(s)=0$, then $\l'(s)\neq0$. Hence the set
$\{s\in I|\ \l\neq0\}$ is dense in $I$. So \eqref{rtriplep-e1} shows
\be \label{rtriple-e2}
r^{\prime \prime \prime} + ( n - 1)   \frac{ r^\prime r^{\prime \prime} }{r}  +
    \frac{R}{n-1} r^\prime \equiv0
\ee
    in $I$. Multiplying \eqref{rtriple-e2} by $r^{n-1}$ and using the fact that $R$ is a constant and $r>0$, we conclude from \eqref{rtriple-e2} that
    \begin{equation*}
 \lf[ r^{n-1} r^{\prime \prime} + \frac{R}{n(n-1)} r^n \ri]^\prime = 0 ,
\end{equation*}
which is  equivalent to
\begin{equation*}
  r^{\prime \prime} + \frac{R}{n(n-1)} r = a r^{1-n}
\end{equation*}
for some constant $ a $. Now \eqref{eq-rRa-s2-2} follows directly from \eqref{eq-rRa-s2}, \eqref{eq-rRh-s2} and the fact $ R(h) = ( n-1)(n-2) \kappa_0 $.

Conversely, suppose $(N,h)$ is Einstein with $\Ric(h)=(n-2)\kappa_0 h$ and the functions $r$, $\l$ satisfy \eqref{eq-rRa-s2}-\eqref{lambda-e1}. Let $g=ds^2+r^2h$. By Lemma \ref{warpedproductmetric-l1}, the scalar curvature $R(g)$ of $g$
is given by
\be\label{warpedproductmetric-e5}
R(g)  = - 2 ( n -1 ) \lf( \frac{ r^{\prime \prime} }{r} \ri)
+ \frac{(n-1)(n-2)\kappa_0}{r^2} - ( n-1) ( n -2) \left(\frac{ r^\prime}{r} \right)^2 .
\ee
Hence, $R(g)=R$ by \eqref{eq-rRa-s2} and \eqref{eq-rRa-s2-2}.
Next, suppose $X, Y\in TN$. By Lemma \ref{warpedproductmetric-l1}
and \eqref{eq-rRa-s2}-\eqref{lambda-e1},  we have
\begin{equation}\label{warpedproductmetric-e4}
\begin{split}
& \l\Ric(g)(X,Y) -\frac{R\l+1}{n-1}g(X,Y) \\
= &
\lf[\frac{(n-2)\l\kappa_0}{r^2}-(n-2)\l\lf(\frac{r'}r\ri)^2-\frac{r''\l}r-\frac{R\l+1}{n-1}\ri]g(X,Y)\\
= & \frac{r'\l'}r g(X,Y) =  \nabla^2_g\l(X,Y)
\end{split}
\end{equation}
and
\begin{equation}\label{warpedproductmetric-e3}
\nabla^2_g\l(\p_s,X)=0=\l\Ric(g)(\p_s,X)-\frac{R\l+1}{n-1}g(\p_s,X).
\end{equation}
On the other hand, differentiating \eqref{eq-rRa-s2},
\eqref{lambda-e1} and canceling $ r'''$, we have
\be \label{rtriplep-e3}
r' \l'' + \lf[ ( n - 1 )a r^{-n} + \frac{R}{n(n-1)} \ri] r' \l
= - \frac{ r'}{n-1}.
\ee
By \eqref{lambda-e1}, if $r'(s)=0$, then $r''(s)\neq0$. Hence the set
$\{r'(s) \in I|\ \l\neq0\}$ is dense in $I$. So \eqref{rtriplep-e3} implies
\be \label{rtriplep-e4}
 \l'' + \lf[ (  n - 1  ) a r^{-n} + \frac{R}{n(n-1)} \ri] \l
= - \frac{ 1}{n-1}.
\ee
By \eqref{eq-rRa-s2},   \eqref{rtriplep-e4} becomes
\be
 \l'' + \lf[ ( n- 1 ) \frac{r''}{r} +   \frac{R}{n-1} \ri] \l
= - \frac{ 1}{n-1} ,
\ee
from which we see that
\begin{equation}\label{warpedproductmetric-e2}
\nabla^2_g\l(\p_s,\p_s)=\l\Ric(g)(\p_s,\p_s)-\frac{R\l+1}{n-1}g(\p_s,\p_s).
\end{equation}
by Lemma \ref{warpedproductmetric-l1}.

By \eqref{warpedproductmetric-e4}, \eqref{warpedproductmetric-e3}
and \eqref{warpedproductmetric-e2}, we conclude that
$\l$ satisfies \eqref{critical-inte-s2}. This completes the proof of the proposition.

\end{proof}

\begin{rmk}\label{s2-r1}
The constant $ a $ in \eqref{eq-rRa-s2}  has a geometric interpretation. Assuming $ r $ and $ (N, h)$ satisfy
$ \mathrm{(i)} $ and $ \mathrm{(ii)} $ in
Proposition \ref{prop-1-s2}, then it follows
from Lemma \ref{warpedproductmetric-l1} and
\eqref{eq-rRa-s2}  that
\be
\begin{split}
 \Ric({g} )(\p_s, \p_s)  = &  - ( n -1 ) a r^{-n} +  \frac{R}{n} \\
 \Ric(g) |_{ T N} =  & \lf( a r^{-n} + \frac{R}{n} \ri) g |_{TN} .
\end{split}
\ee
Hence, $ a = 0 $ if and only if $ g $ is an Einstein metric.
\end{rmk}

\begin{rmk}\label{s2-r2}
The condition \eqref{eq-rRa-s2} on the function $ r $ in Proposition
\ref{prop-1-s2} turns out to be the same condition that Kobayashi obtained in \cite{Kobayashi}
where he constructed warped-product solutions to an equation,
similar to \eqref{critical-inte-s2},
\be \label{eq-gstatic-s2}
 -(\Delta_g f ) g+\nabla^2_g f  - f  \Ric(g)  = 0 ,
\ee
where the metric $ g $ and the function $ f $ are the unknowns.
Kobayashi proved that, if $(N, h)$ has constant sectional curvature, then $ g = d s^2 + r^2 h $
satisfies \eqref{eq-gstatic-s2} with some function $ f = f(s) $ if
and only \eqref{eq-rRa-s2} holds $\mathrm{(}$see Lemma 1.1
in \cite{Kobayashi}$\mathrm{)}$. Equation
\eqref{eq-gstatic-s2} is of interest to study because of its root in
 general relativity $ \mathrm($see \cite{FischerMarsden1975}, \cite{KobayashiObata1981}, \cite{Corvino07}, etc$\mathrm{)}$.
\end{rmk}

Next, we consider the function $ \l $ in Proposition \ref{prop-1-s2}.
Viewed as an ODE about $ \l $,  equation \eqref{lambda-e1}
becomes singular at points where $ r^\prime $ is zero.
Nonetheless, we show  it always has a solution $ \l $
as long as $ r $ is a
non-constant solution to \eqref{eq-rRa-s2}.

\begin{lma} \label{lma-3-s2}
Suppose $ r$ is a smooth, positive, non-constant
solution to
\be \label{eq-rRa2-s2}
   r^{\prime \prime} + \frac{R}{n(n-1)} r = a r^{1-n}
\ee
on $ I$, where $ R $ and $ a $ are some given
constants. Then
\begin{enumerate}
\item[(i)] $r'$ and $r''$ can not vanish simultaneously at any point
in $ I $.
\item[(ii)] Suppose  $r'(s_0)\neq 0$, $s_0\in I$. Given any initial condition $\l(s_0)=c$, there is unique solution $\l$ of \eqref{lambda-e1} on $I$ such that $\l(s_0)=c$.
\item[(iii)]   Suppose  $r''(s_0)\neq 0$, $s_0\in I$. Given any initial condition $\l'(s_0)=c$, there is unique solution $\l$ of \eqref{lambda-e1} on $I$ such that $\l'(s_0)=c$.
\item[(iv)] Any two solutions to \eqref{lambda-e1}
differs by a constant multiple of $ r^\prime $.
\end{enumerate}
\end{lma}
\begin{proof}
(i) Taking derivative of \eqref{eq-rRa2-s2},
\be\label{eq-drRa-s2}
r^{\prime \prime \prime} +
\lf[ \frac{R}{n(n-1)} +  ( n -  1) a r^{-n} \ri]
 r^\prime = 0 .
\ee
 Suppose $r'(s_0)=r''(s_0)=0$ for some $s_0\in I$, then $r'\equiv0$ by the uniqueness of solutions to the ODE \eqref{eq-drRa-s2}. Since $r$ is non-constant, this is impossible.

 (ii) Suppose $r'(s_0)\neq0$ and $c$ is given. On $ I $,
 we can solve for $\l$
 \be \label{eq-r2d-s2}
 \l^{\prime \prime} +
\lf[ \frac{R}{n(n-1)} + (n-1) a r^{-n} \ri] \l =  - \frac{1}{n-1}
  \ee
  with initial data $\l(s_0)=c$ and
  $\l'(s_0) = \frac{1}{ r'(s_0)} \lf[cr''(s_0)- \frac{r(s_0)}{n-1}\ri] .$
  Let $ \l $ be such a solution to \eqref{eq-r2d-s2}.
  By \eqref{eq-drRa-s2} and \eqref{eq-r2d-s2}, we have
  $$
  \lf(r'\l'-r''\l+\frac{r}{n-1}\ri)'=0
  $$
  on $I$. Since $r'\l'-r''\l+\frac{r}{n-1}=0$ at $s_0$, $\l$ satisfies \eqref{lambda-e1} with $\l(s_0)=c$.

Conversely, if $\l$ is a solution of \eqref{lambda-e1} with $\l(s_0)=c$, we must have $\l'(s_0)= \frac{1}{ r'(s_0)}
\lf[cr''(s_0)- \frac{r(s_0)}{n-1}\ri] $ since
$ r'(s_0) \neq 0$. On the other hand,
 $ \l $ satisfies \eqref{eq-r2d-s2} by the proof of
 Proposition \ref{prop-1-s2}. Hence,
 $ \l $  is unique.

  (iii) can be proved in the same way as (ii) is proved.

  (iv) Let $ \l_1 $, $\l_2 $ be any two solutions to \eqref{lambda-e1}  on $ I$. Let $ \phi = \l_1 - \l_2 $, then $ \phi $ satisfies
$
r^\prime \phi^\prime - r^{\prime \prime} \phi = 0 ,
$
which implies $ \phi $ is a constant multiple of $ r^\prime $
on any sub-interval of $ I$ where $ r^\prime $ is never zero.
By (i), the roots of $ r^\prime $ are isolated in $ I $. Therefore,
$ \phi = C r^\prime $ on $ I$ for some constant $ C$.

\end{proof}

In what follows, we always assume $ R $ and
$ a $ are two  given constants.
By Proposition \ref{prop-1-s2} and
Lemma \ref{lma-3-s2},
  any non-constant, positive solution $ r $ to the ODE
\be \label{eq-rRa3-s2}
 r^{\prime \prime} + \frac{R}{n(n-1)} r =a r^{1-n},
 \ee
on an interval $ I $,  will give rise to a
metric $ g = d s^2 + r^2 h $, on $ M = I \times N$,
which satisfies \eqref{critical-inte-s2}
for some function $ \l$ (provided $(N, h)$ is an
Einstein manifold with Ricci curvature properly chosen).
It is natural to know if one can obtain a compact $(M,g) $
from this procedure  such that $ \l = 0 $ on $ \partial M $.
For this purpose, we  consider
solutions $ r $ to \eqref{eq-rRa3-s2} existing on $ \R^1 $ and
ask how many roots the associated solutions $ \l $ to
\eqref{lambda-e1}  may have.

The following lemma was proved by Kobayashi in
 \cite{Kobayashi}.

\begin{lma}\label{lma-kb-s2}
Suppose $ a > 0$, then any local positive solution to
\eqref{eq-rRa3-s2} can be extended as a positive solution
  on $ \R^1 $. If in addition $ R > 0$,  then each non-constant solution
  on $ \R^1 $ is periodic.
\end{lma}

For reasons which will be clear in Lemma   \ref{CF-l3}, we
impose the assumption $ a > 0 $ hereafter. For any
positive solution $ r $ to \eqref{eq-rRa3-s2} on $ \R^1 $,
there exists a constant $ \kappa_0 $ such that
\be \label{r-inte-const}
( r')^2 +  \frac{R}{n(n-1)} r^2 +
\frac{2a}{n-2} r^{2-n} = \kappa_0 .
\ee
As $ a > 0 $, it follows directly from \eqref{r-inte-const} that $ r $ is  bounded from below by a positive constant.

\begin{lma}\label{lma-kb-s2-l2}
Suppose $a>0$. Let $r$ be a non-constant, positive solution to \eqref{eq-rRa3-s2} on $ \R^1$. If $R\le 0$, then $ r'(s) $ has
a unique root. If $R>0$, then $r'(s)=0$ if and only if $r(s)$ is the maximum or the minimum of $r$.

\end{lma}
\begin{proof}
Suppose $R\le 0$, then \eqref{eq-rRa3-s2}  implies
$
r''\ge ar^{1-n}.
$
Assume $r'>0$ everywhere, then $r(s)\le r(0)$ for all $s\le 0$.
So $r'' (s) \ge C$ for some positive constant $C$ for $s<0$. This
 implies $r'(s)<0$ somewhere, which is a contradiction. Similarly, it is impossible to have $r'<0$ everywhere. Hence $r'(s)=0$ for some $s$. Since $r''>0$, the root of $ r' (s) $ is unique.

Suppose    $R> 0$, then $r$ is periodic by Lemma \ref{lma-kb-s2}. Let $r_{\max}$ and $r_{\min}$ be the maximum and minimum of $r$.
If   $ r'(s_0)=0 $, then \eqref{r-inte-const} implies
\be \label{eq3}
 \frac{R}{n(n-1)} r^2(s_0) +
\frac{2a}{n-2} r^{2-n}(s_0) = \kappa_0
\ee
with $ \kappa_0 > 0$.
In particular,  \eqref{eq3} holds with
$ r(s_0)$ replaced by $ r_{\max} $ or $ r_{\min}$.
Consider
\be
F ( r ) = \frac{R}{n(n-1)} r^2 +
\frac{2a}{n-2} r^{2-n}
\ee
as a function of $ r $. Then
\be \label{doffr}
\frac{d F}{d r} = 2 r \lf[ \frac{R}{n(n-1)}  - a r^{ -n } \ri] .
\ee
Let $ r_0  = \lf( \frac{n(n-1)a}{R} \ri)^\frac{1}{n} $,
then $ F(r) $ is strictly decreasing on $(0, r_0 )$ and
strictly increasing on $(r_0, \infty)$.
So for then give $ \kappa_0 > 0 $, \eqref{eq3} at most has
$ 2 $ distinct solutions for $ r(s_0) $. Hence, $ r(s_0) $ is one of
$ r_{\min} $ and $ r_{\max} $. Moreover, as $ r $ is assumed not to
be a constant, we have
\be \label{rminrmax}
 r_{\min} < r_0 < r_{\max}.
 \ee

\end{proof}
Let $r$ be given as in Lemma \ref{lma-kb-s2-l2}. Without losing generality,  we may assume  $r'(0)=0$. By the uniqueness
of solutions of ODEs, $r$ is  an even function.
In case $R>0$ and $r$ is non-constant, 
the roots of $ r'(s) $ form a discrete subset in $ \R^1$. If 
we arrange so that $r(0)=r_{\min}$ (or $r_{\max}$) and if
$0, \pm s_1,\pm s_2,\dots$ are zeros of $r'$ with $s_1<s_2<\dots$, then $r(\pm s_1)=r_{\max}$ (or $r_{\min}$ respectively) and $r$ is periodic with period $s_2$.
Now let $\l_0$ be the solution of \eqref{lambda-e1}
on $ \R^1 $ with $\l_0'(0)=0$, which exists  and is unique by Lemma \ref{lma-3-s2}, then $\l_0$ is also an even function.

 \begin{prop}\label{examples-p1}
Let   $ a > 0 $ and $ R $ be two constants. Let
$ r $ be a  positive, non-constant solution
to \eqref{eq-rRa3-s2} on
$ \R^1 $ satisfying  $ r^\prime (0) = 0 $.
For  such a given $ r $,
let $\l_0$ be the solution to \eqref{lambda-e1}
on $ \R^1 $ satisfying $\l_0'(0)=0$.
Let $ \l $ be another solution
 to \eqref{lambda-e1} on $ \R^1 $.
 By Lemma \ref{lma-3-s2},
 $ \l = \l_0 + C r' $ for some constant $ C$.

 \begin{enumerate}
                 \item[(i)] Suppose $R=0$. Then $\l(0)>0$, 
                 $ \int_{1}^{+\infty} \frac{ r}{ ( r^\prime )^2 } \ d \tau
                  = + \infty$,                  
                  $\l$ has a unique positive root $\zeta_1 $ and a unique negative root $ \zeta_2$ and they are related by
                     \begin{equation}\label{lambda0-e2}
                    \int_{\theta}^{\zeta_1 }\frac{r}{(r')^2}d \tau=\int_{-\theta}^{\zeta_2 }\frac{r}{(r')^2}d \tau ,
                     \end{equation}
                     where $ \theta $ and $-\theta$
                     are the unique positive and negative roots of $\l_0$.
                 \item [(ii)] Suppose $R<0$. Then $\l(0)>0$, 
              $ \int_1^{+ \infty} \frac{r}{(r')^2}d\tau< + \infty $,   
              $\l_0$ has a  unique positive root  $\theta $ and a unique negative root $ -\theta$. Moreover, 
(a) if $C\le - \frac{1}{n-1} \int_\theta^{+ \infty} \frac{r}{(r')^2}d\tau$,
 then  $\l$ has a unique root and the root is positive; (b) if $C\ge \frac{1}{n-1} \int_\theta^{+ \infty} \frac{r}{(r')^2}d\tau$, then $\l$ has a unique  root  and the root is negative; (c) if $ | C | < \frac{1}{n-1} \int_\theta^{+ \infty} \frac{r}{(r')^2}d\tau$,  then $\l$ has a unique positive root $\zeta_1$ and a unique negative root $\zeta_2$ and 
 $ \zeta_1$, $ \zeta_2 $ 
 are related by  \eqref{lambda0-e2}.
In particular, $ \zeta_1 > \zeta $ and $ \zeta_2 < -\zeta$, where
$ \zeta  \in (0, \theta) $ 
is the constant determined by 
\be
 \int_{\zeta}^{ \theta} \frac{r}{(r')^2}d\tau = 
 \int_\theta^{+ \infty} \frac{r}{(r')^2}d\tau .
  \ee

                 \item [(iii)] Suppose $R>0$.  Then $\l$  has exactly one root between any two consecutive roots of $r'$. If
                 $r(0)=r_{\min}$ {\em(}respectively $r_{\max}${\em)},
                 then $\l(0) >0$ {\em (}respectively $ < 0${\em)}. Let $\theta>0$ be the first positive root of  $\l_0$. Then the smallest positive root $\zeta_1$ and the largest negative root of $\zeta_2$ of $\l$ are related by \eqref{lambda0-e2}.
                                \end{enumerate}
\end{prop}
\begin{proof} Since $ r' (0) = 0 $, by \eqref{lambda-e1} we have
$ r'' (0) \l (0) = \frac{r(0)}{n-1} $.
In particular, $ \l(0) $ and $ r'' (0) $ have the same sign.

(i) Suppose $ R = 0 $. We have $ r'' = a r^{1-n} > 0 $ for all $ s $.
Hence, $ \l (0) > 0 $.
On $(0, + \infty)$, the function
$$
-\frac{r'}{n-1}\int_1^s\frac{r}{(r')^2}d\tau
$$
is a solution to \eqref{lambda-e1}. By Lemma \ref{lma-3-s2},
we have
\begin{equation}\label{lambda0-e1}
\l(s)=r'(s)\lf(C_1 - \frac{1}{n-1}\int_1^s\frac{r}{(r')^2}d\tau\ri)
\end{equation}
for some constant $C_1 $ for any $ s > 0 $. Let $\kappa_0>0$ be the constant in  \eqref{r-inte-const} with $R=0$.
Then $(r')^2<\kappa_0 $ and $r(s)\ge r(0)>0$. Hence,
\be \label{eq-intr1-s2}
  \lim_{ s \rightarrow + \infty} \int_{1}^s \frac{ r}{ ( r^\prime )^2 } \ ds
  = + \infty .
  \ee
Since $r'(0)=0$ and $r(0)>0$, we also have
\be \label{eq-intr2-s2}
  \lim_{ s \rightarrow 0} \int_{1}^s \frac{ r}{ ( r^\prime )^2 } \ ds
  = - \infty .
  \ee
By \eqref{lambda0-e1}-\eqref{eq-intr2-s2},
 we conclude that $\l$ has a unique positive root $\zeta_1$. Similarly, we can prove that $\l$ has a unique negative root $\zeta_2$.

Let $\theta>0$ be the unique positive root of $\l_0$, then $-\theta$ is its negative root because $\l_0$ is an even function.
Moreover,  \eqref{lambda-e1} implies
\begin{equation} \label{lzformz}
\l_0(s)= \left\{
       \begin{array}{ll}
         -\frac{r'(s)}{n-1}\int_\theta^s\frac{r}{(r')^2}d\tau, & \hbox{for $s>0$;} \\
          -\frac{r'(s)}{n-1}\int_{-\theta}^s\frac{r}{(r')^2}d\tau , & \hbox{for $s<0$.}
       \end{array}
     \right.
\end{equation}
Therefore,
\begin{equation} \label{lformz}
\l(s)= \left\{
       \begin{array}{ll}r'(s)
        \lf(C -\frac{1}{n-1}\int_\theta^s\frac{r}{(r')^2}d\tau\ri), & \hbox{for $s>0$;} \\
         r'(s)\lf(C -\frac{1}{n-1}\int_{-\theta}^s\frac{r}{(r')^2}d\tau\ri) , & \hbox{for $s<0$.}
       \end{array}
     \right.
\end{equation}
Since $\l(\zeta_1)=\l(\zeta_2)=0$, \eqref{lambda0-e2} follows
from  \eqref{lformz}.

(ii) Suppose $R<0$. Using the fact
$ r (s) \geq r(0) > 0 $, we have
$ r'' = a r^{1-n} - \frac{R}{n(n-1)} r \geq \alpha > 0  $
for some  constant $ \alpha $.  In particular, this implies
$ \l (0) > 0 $, and $r(s)\ge \beta s^2$ for some $\beta>0$ for all $s> 0$ sufficiently large.
By \eqref{r-inte-const}, $r^2/(r')^2$ is bounded. Hence, 
$$
\int_1^{+ \infty} \frac{r}{(r')^2}d\tau< + \infty .
$$
 Similar to the proof in (i), we know
 there exists a constant $ C_0 $ such that
\begin{equation}\label{lambda0-e3}
\l_0(s) = r'(s)\lf( - C_0 +\frac1{n-1}\int_s^{+ \infty} \frac{r}{(r')^2}d\tau\ri)
\end{equation}
for $ s > 0$.
By the L'H\^{o}pital rule,   \eqref{eq-rRa3-s2} and the facts
$ \lim_{s \rightarrow + \infty} r'(s) = + \infty$ and
$ \lim_{s \rightarrow + \infty} r(s) = + \infty$,
we have
\be  \label{eq-lopital-s2}
\lim_{s \rightarrow +\infty}
  \frac{r^\prime(s)}{n-1}
\int_{s}^{+\infty} \frac{ r}{ ( r^\prime )^2 } \ ds
=  \frac{1}{n-1} \lim_{s \rightarrow + \infty} \frac{ r(s) }{ r^{\prime \prime} ( s) }
=  \frac{ n  }{ - R }.
\ee
On the other hand,
\be \label{eq-lzero-s2}
\l_0 (0) = \frac{1}{n-1} \frac{r(0)} { r^{\prime \prime}(0)  }
= \frac{ 1}{ (n-1) a r^{-n} - \frac{R}{n} } <  \frac{n}{-R}.
\ee
Suppose  $C_0 \leq 0$. Then
it follows from \eqref{lambda0-e3}-\eqref{eq-lzero-s2}
and the fact $ \l_0$ is even
that $\l_0 +n/R$ has an interior negative minimum. This is impossible because, by the proof in Proposition \ref{prop-1-s2},
$ \l_0 $ satisfies \eqref{rtriplep-e4} or equivalently
$ \l_0 + \frac{n}{R} $ satisfies
$$
\lf(\l_0 +\frac{n}{R}\ri)''+\frac{R}{n(n-1)}\lf(\l_0 +\frac{n}{R}\ri)
=-\l(n-1)ar^{-n}\le 0.
$$
Therefore $C_0 > 0$. In particular, $ \lim_{ s \rightarrow + \infty}
\l_0(s) = - \infty $.
 Since $\l_0(0)>0$
and $ \l_0 $ is even, we conclude from
 \eqref{lambda0-e3} that $\l_0$ has a unique positive root $\theta$
and a unique negative root $-\theta$. Moreover, $ \theta $ and
$ C_0 $ are related by
$$ C_0 = \frac{1}{n-1} \int_\theta^{+\infty} \frac{r}{(r')^2}d\tau . $$
Now let $ \l  = \l_0 + C r'$ be another solution, then
\begin{equation} \label{lformn}
\l(s)= \left\{
       \begin{array}{ll}r'(s)
        \lf(C-C_0 +\frac{1}{n-1}\int_s^\infty\frac{r}{(r')^2}d\tau\ri), & \hbox{for $s>0$;} \\
         r'(s)\lf(C +C_0-\frac{1}{n-1}\int_{-\infty}^{s} \frac{r}{(r')^2}d\tau\ri) , & \hbox{for $s<0$.}
       \end{array}
     \right.
\end{equation}
It follows from \eqref{lformn} that (a) if $C\le -C_0$, $\l$ has a unique  root and the root is positive;
(b) if $C\ge C_0$, $\l$ has a unique root and the root is negative;
(c) if $ | C | <C_0$,  $\l$ has a unique positive root $\zeta_1$ and a unique negative root $\zeta_2$ and $ \zeta_1$, $\zeta_2 $ 
satisfy \eqref{lambda0-e2}; moreover, 
\eqref{lambda0-e2} implies that 
 \begin{equation}\label{zeta-e2}
                    \int_{\theta}^{\zeta_1 }\frac{r}{(r')^2}d \tau
                    > 
        \int_{-\theta}^{- \infty }\frac{r}{(r')^2}d \tau             
                    = - \int_{ \theta}^{ \infty }\frac{r}{(r')^2}d \tau
                    = \int_{\theta}^\zeta \frac{r}{(r')^2}d \tau .
                      \end{equation}
Therefore, $ \zeta_1 > \zeta $. Similar, we have
$ \zeta_2 < - \zeta $.


(iii) Suppose $ R > 0 $. Let  $ \{ s_k \} $ be the increasing
positive sequence such that $ \{ 0, \pm s_1, \pm s_2, \ldots \}$
is the set of roots of $ r' (s) $.  By \eqref{lambda-e1},
 $ \l ( s_k ) $ (or $\l(- s_k )$) has the same sign as
$ r''(s_k) $ (or $r''( - s_k) $).
Suppose $r(0)=r_{\min}$. Then $r( s_1)=r_{\max}$ and $r' > 0$ in $(0,s_1)$. Moreover,
we have $r''(0) > 0$ and $r''(s_1) < 0$, which imply $ \l(0) > 0 $
and $ \l (s_1) < 0 $. Hence, $ \l (\zeta_1) = 0 $ for some
$ \zeta_1 \in (0, s_1)$. By \eqref{lambda-e1}, we  have
 \begin{equation}\label{lambda0-e4}
\l(s)= -\frac{ r'(s)}{n-1}\int^{s}_{\zeta_1} \frac{r}{(r')^2}d\tau
\end{equation}
for  any $s \in (0,s_1)$, which shows
$ \zeta_1 $ is the unique root of $ \l $ in  $(0,s_1)$.
Similar arguments prove that $\l$ has a unique root between
any two consecutive roots of  $r'$. Let $ \zeta_2 $ be the
maximum negative root of $ \l $. The claim that $ \zeta_1 $
and $ \zeta_2 $ satisfy \eqref{lambda0-e2}  follows from
the same proof as in  (i) and (ii). The case $ r(0) = r_{\max}$
can proved similarly.
\end{proof}

Now we are in a position to construct compact manifolds
with boundary with a non-Einstein critical metric.

\vskip .2cm

{\bf Examples}:\vskip .2cm

{\bf (1)} Given  $ a > 0 $ and $ R $ two constants, let
$ r $ be a positive solution to \eqref{eq-rRa-s2} on $ \R^1$
satisfying $ r'(0) = 0 $. Let $ \kappa_0 $ be an  integral
 constant of \eqref{eq-rRa-s2} so that \eqref{eq-rRa-s2-2}  holds
 for  $ r $.  Let $(N, h)$ be an $ (n-1)$-dimensional, connected, closed Einstein manifold satisfying
 $\Ric(h)=(n-2)\kappa_0h$.
 We note that $ \kappa_0 $ must be positive if $ R \geq 0$
 and  $ \kappa_0 $ can be arbitrary if $ R < 0 $.
  Let $\l_0$ be the solution to \eqref{lambda-e1} on $ \R^1 $
 with $\l_0'(0)=0$.
 Let $ \theta $ and $ - \theta $ be the unique positive and
 negative roots of $ \l_0 $.
 Let $\zeta_1 > 0 $ and $\zeta_2 < 0 $ be chosen such that \eqref{lambda0-e2} holds. Define $ I = [ \zeta_2, \zeta_1 ]$. Then
 $ (\Omega, g) = (I \times N,  d s^2 + r^2 h) $
 satisfies \eqref{critical-inte-s2} for some
 $ \l $ vanishing on $ \p \Omega $.
In this case, $ g $ has constant scalar curvature $ R $ and
$ \p \Omega $  has two connected components.

 \vskip .2cm

{\bf (2)}  Let $ I $ and $ (\Omega, g)$ be given as in {\bf (1)}
 with  $\zeta_1=\theta$ and $ \zeta_2=-\theta$.
 Suppose $G$ is a finite subgroup of isometries of $(N, h)$
 which acts freely on $ N$.  Consider the action of
 $G\times \mathbb{Z}_2$ on $\Omega$  defined by
$$
( \alpha, k)(s,x)=((-1)^k s, \alpha(x)),
$$
where $ \alpha \in G $ and $ k \in \mathbb{Z}_2 = \{ 0, 1 \}$.
This is an action of isometry on $(\Omega, g)$.
Suppose $ H $ is a subgroup of $ G \times \mathbb{Z}_2 $
which does not contain $({\bf id}, 1)$, where
${\bf id} $ denotes the identity map on $ N$.
If $ (\alpha, 1) \in  H$, then
$ (\alpha, 0) \notin H $ for
otherwise  $ ({\bf id}, 1) = (\alpha, 1) ( \alpha^{m-1}, 0) $
would be in $ H $ (here $ m $ is the order
of $ \alpha $ in $ G$). From this it can be  easily checked
that $ H$ acts freely on $ \Omega $. Since $ (\Omega, g) $
is compact with boundary, so is the quotient manifold
$(\Omega, g) / H $.
The function $ \l_0 $  descends to a function $ \l $
on $(\Omega, g) / H $ which satisfies  \eqref{critical-inte-s2}
and vanishes on the boundary $ \p  \lf( \Omega/ H \ri) $.

If $ H \neq H \cap ( G \times \{ 0 \} ) $,
we  claim that $ \p  \lf( \Omega/ H \ri) $ is connected.
 To see this, let $ \pi $ be  the natural projection map from
 $ \Omega $ to $ \Omega  / H $.
 Then
 $$\p  \lf( \Omega/ H \ri) =  \pi ( \p \Omega ) = \pi( \{ \theta \} \times N ) \cup \pi ( \{ - \theta \} \times N ) .$$
 Suppose $ (s, x) \in \p \Omega$,  say $ s =  \theta $, then
$ \pi ( \theta, x) = \pi( - \theta, \alpha (x) ) $, where $ (\alpha, 1) $ is
an element in $ H $ but not in $ H \cap ( G \times \{ 0 \} )$.
Hence,
$$ \pi( \{ \theta \} \times N ) \cap \pi ( \{ - \theta \} \times N )
\neq \emptyset ,$$
which implies  $ \p  \lf( \Omega/ H \ri) $ is connected.
In the special case when $ (N, h)$ admits an isometry  $ \alpha $ without fixed points  so that $\alpha^2=  {\bf id}$, we can
take $ G = \{ {\bf id}, \alpha \} $ and
$ H = \{ ({\bf id}, 0 ), ( \alpha, 1) \} $.
 Then $(\Omega, g) / H $ has a connected boundary that
 is isometric to a constant re-scaling of $ (N, h)$.

 \vskip .2cm

In the above construction, suppose $ R \leq 0 $, $ r $ is 
chosen such that $ \kappa_0 = 1 $ and $ (N, h )$ is taken
to be  $\mathbb{S}^{n-1}$, then $  g =  d s^2 + r^2 h $ 
is simply the usual spatial Schwarzschild metric or 
Ads-Schwarzschild metric, whose mass is given by 
the constant  $ a $. To see this, one can  make
a change of variable $ s = s (r) $ and use  \eqref{r-inte-const}
to re-write $ g $  as 
  \be \label{eq-sg-s2}
  g =  \frac{1}{ 1 - \frac{R}{n(n-1)} r^2  -  \frac{2a}{n-2} r^{2-n} }d r^2   + r^2 d h . 
  \ee
Note that the antipodal map $ \alpha $ on $ \mathbb{S}^{n-1} $ 
is an isometry  without fixed points such that 
$\alpha^2=  {\bf id}$. Hence, the following 
results follow directly from the above construction and 
Proposition \ref{examples-p1} (i) and (ii).

\begin{cor}  \label{Schwarzschild}
Let $ (M, g)$ be a complete, spatial Schwarzschild
manifold 
with positive mass. 
Let $ \Sigma_0 $ be the horizon in $(M, g)$ (i.e. the unique 
closed minimal surface in $(M,g)$.) The followings are true: 
\begin{enumerate}

\item[(i)] There exist  functions $ \l $ on $ (M, g)$ satisfying \eqref{critical-inte-s2}.

\item[(ii)] Let $ M_+ $, $ M_- $ be the two components of  $ M \setminus \Sigma_0 $.  Then for any rotationally symmetric sphere $ \Sigma_{\zeta_1} $ in $ M_+$, there exists a rotationally symmetric sphere  $ \Sigma_{\zeta_2}$ in $ M_-$ such that $ g $ is a critical metric on the (closed) domain  $ \Omega $ bounded by $ \Sigma_{\zeta_1}$ and $ \Sigma_{\zeta_2} $.

\item[(iii)] There exists a rotationally symmetric sphere $ \Sigma_{\theta} $ in $ M_+ $ (or equivalently $ M_-$) such that if $ \Omega $ is the (closed) domain bounded by $ \Sigma_\theta $ and $ \Sigma_0 $ and if $ (\tilde{\Omega}, \tilde{g}) $  is the quotient manifold obtained from $ (\Omega, g) $ by identifying 
points on $ \Sigma_0 $  through the antipodal map on 
$ \Sigma_0 $, then $ \tilde{g} $ is a critical metric on 
$ \tilde{\Omega}$. 

\end{enumerate}
\end{cor}

\begin{cor}  \label{Ads-Schwarzschild}
Let $ (M, g)$ be a complete, spatial Ads-Schwarzschild
manifold with positive mass. 
Let $ \Sigma_0 $ be the horizon in $(M, g)$ (i.e. the unique 
closed minimal surface in $(M,g)$.)
The followings are true:
\begin{enumerate}

\item[(i)] There exist  functions $ \l $ on $ (M, g)$ satisfying \eqref{critical-inte-s2}.

\item[(ii)]  Let $ M_+ $, $ M_- $ be the two components of  $ M \setminus \Sigma_0 $.  There exists a  rotationally symmetric sphere $ \Sigma_\zeta $ in $ M_+ $ and a  rotationally symmetric sphere $ \Sigma_{-\zeta} $ in $ M_- $, which is the image of $ \Sigma_\zeta$
under the reflection 
with respect to $ \Sigma_0 $,
such that if $ U $ is the (closed) domain bounded by $ \Sigma_\zeta $ and
$ \Sigma_{ - \zeta} $, then for any rotationally symmetric sphere $ \Sigma_{\zeta_1} $ in $ M_+ \setminus {U} $, there exists a rotationally symmetric sphere  $ \Sigma_{\zeta_2}$ in $ M_- \setminus
{U}$ such that $ g $ is a critical metric on the (closed) domain $ \Omega $ bounded by $ \Sigma_{\zeta_1}$ and $ \Sigma_{\zeta_2} $.


\item[(iii)] Let $ U $ be  as in (ii). 
There exists a rotationally symmetric sphere $ \Sigma_\theta$ in 
$ M_+ \setminus {U}$ (or equivalently $ M_- \setminus {U} $) such that if $ \Omega $ is the (closed) domain bounded by $ \Sigma_\theta $ and $ \Sigma_0 $ and if $ (\tilde{\Omega}, \tilde{g}) $  is the quotient manifold obtained from $ (\Omega, g) $ by identifying 
points on $ \Sigma_0 $  through the antipodal map on 
$ \Sigma_0 $, then $ \tilde{g} $ is a critical metric on 
$ \tilde{\Omega}$. 

\end{enumerate}
\end{cor}

We end this section by a discussion on the sign of
the first Dirichlet eigenvalue of $(n-1) \Dg + R $
of those examples
constructed in {\bf (1)} and {\bf (2)} with $ R > 0$.

\begin{prop}\label{1steigenvalue-p1}
Let $R>0$, $a>0$ and $r$ be given as in
Proposition \ref{examples-p1}(iii).
Suppose $-s_1$, $0$ and $s_1$ are three consecutive roots of $r'$. Let $I$ be a finite closed interval in $\R^1$. Consider the manifold $\Omega=I\times N$ with the metric $g=ds^2+r^2h$, where $h$ is an Einstein metric on a closed manifold $N$ such that $\Ric(h)=(n-2)\kappa_0h$ with $\kappa_0$ satisfying \eqref{eq-rRa-s2-2}.

 \begin{itemize}
   \item [(i)] If $ [0,s_1]$ is a proper subset of $I$,
   then the first Dirichlet eigenvalue of $(n-1)\Dg+R$
   on $(\Omega, g)$ is negative.
   \item [(ii)] Let $\l$ be a solution to \eqref{lambda-e1} on
   $ \R^1$. Let $\zeta_2 \in (-s_1,0)$ and $\zeta_1 \in (0,s_1)$ be
   the two consecutive roots of $\l$.
   Let $I= [ \zeta_2,\zeta_1]$.
   Then the first Dirichlet eigenvalue of $(n-1)\Dg+R$ on $\Omega$ is
    positive if $r(0)=r_{\min}$ and is negative if $ r(0) = r_{\max}$.

    \item[(iii)] Suppose $ r(0) = r_{\min} $.
    Let $ \l_0 $ be the even solution to \eqref{lambda-e1} on
   $ \R^1$.
    Let $ - \theta  \in (-s_1,0)$ and $\theta  \in (0,s_1)$ be
   the two consecutive roots of $\l_0$. Let $ I = [ - \theta, \theta] $.
   Let $ (\Omega, g) / H $ be given as in {\bf(2)}. Then
    the first Dirichlet eigenvalue of
    $(n-1)\Dg+R$ on $(\Omega, g) / H  $  is  positive.

 \end{itemize}
\end{prop}
\begin{proof}(i) Note that \eqref{eq-rRa3-s2} implies
 $$
r''' + ( n -1) \frac{ r' r''  }{ r} + \frac{R}{n-1} {r'} = 0 ,
$$
which  implies
$$
\Dg r'+\frac{R}{n-1}r=0
$$
on $ \Omega $.
Since $r' (0)=r'(s_1) = 0$ and $r'$ does not change sign in $(0,s_1)$, the first Dirichlet
eigenvalue of $(n-1)\Dg+R$ on $(0,s_1)\times N$ must be zero.
As $ (0, s_1) \times N $ is a proper subset of $ \Omega$,
we conclude that  (i) is true (see Lemma 1 in \cite{Fischer-ColbrieSchoen1980}).

(ii) By \eqref{critical-e2-s2}, we have
\begin{equation}\label{lambda0-e5}
\Dg\l+\frac{R}{n-1}\l=-\frac{n}{n-1}
\end{equation}
on $ \Omega $.
Let $ \gamma $ the first eigenvalue of
$ (n-1) \Dg + \frac{R}{n-1} $ on $ \Omega$.
Let $ \phi $ be an eigenfunction satisfying
\be \label{feigeneq}
\lf\{
\begin{array}{rcl}
(n-1) \Dg \phi + \frac{R}{n-1} \phi + \gamma \phi & = & 0 \
\mathrm{on} \ \Omega   \\
\phi & = & 0 \ \mathrm{on} \ \p \Omega
\end{array} .
\right.
\ee
It follows from \eqref{lambda0-e5}-\eqref{feigeneq} and
the fact $ \l = 0 $ on $ \p \Omega$  that
\be  \label{signofgamma}
 \gamma \int_\Omega \l \phi =  \frac{n}{n-1} \int_\Omega \phi .
\ee
Since both $ \phi $ and $ \l $ do not change sign in the interior of $ \Omega$, \eqref{signofgamma} implies that $ \gamma $
has the same sign as $ \l $ on $(-\zeta_2, \zeta_1)$.
If $ r(0) = r_{\min}$, we have $ \l (0) > 0 $ by (iii) in
Proposition \ref{examples-p1}, hence $ \gamma > 0 $.
Similarly, if $ r(0) = r_{\max}$,
 we have $ \l (0) < 0 $  and $ \gamma < 0 $.
 Therefore, (ii) is proved.

(iii) follows directly from (ii) and the fact that the natural
projection map from $(\Omega, g)$ to $(\Omega, g) / H $
is a local isometry.

\end{proof}

\section{conformally flat critical metrics}
In this section, we  consider conformally flat
metrics $ g $ satisfying
\begin{equation}\label{critical-e1-s3}
    -(\Delta_g\lambda)g+\nabla^2_g\lambda-\lambda \Ric(g)
    =g
\end{equation}
for some function $ \l $. Our main goal is to
classify all  compact manifolds with boundary 
which admit a  conformally flat critical metric.

We start with local properties of such a metric.
Similar  to the work  of Kobayashi
and Obata in \cite{KobayashiObata1981}, we have the following:

\begin{lma}\label{CF-l1}
Let $ (\Omega^n, g) $ be a connected, conformally flat
Riemannian manifold.
Suppose  there exists a smooth function $ \l $ such that
$ g $ and $ \l $ satisfy \eqref{critical-e1-s3}.
For $c\in \R$,  let $N$ be a component of $\l_c$  which is
the level set $\{\l=c\}\subset \Omega$
such that $\nabla\l\neq0$ on $N$. Then the following holds:
\begin{itemize}
  \item [(i)] $|\nabla\l|$ is constant on $N$.
  \item [(ii)] $N$ is totally umbilical with constant mean curvature.
  \item [(iii)] $N$ has constant sectional curvature.
\end{itemize}
\end{lma}
\begin{proof}  Let $ R $ be the scalar curvature of $ g $.
By \cite{MiaoTam08}, $ R $ equals a constant.
The proof in \cite{KobayashiObata1981} can then be carried over to our case. For the sake of completeness, we include the relevant details.  First note that it is sufficient to consider the case that $c\neq0$.  Since $\Omega$ is conformally flat, we have (see \cite{Lafontain1985} for example):
\begin{equation}\label{Schouten-e1}
    (\nabla_XS)(Y,Z)-(\nabla_YS)(X,Z)=0
\end{equation}
for all vector fields $X, Y, Z$, where $S$ is the Schouten tensor given by (at the points where $\l\neq0$)
\begin{equation}\label{Schouten-e2}
  \begin{split}
   (n-2) S= & \Ric(g) -\frac{R}{2(n-1)}g \\
   = & \l^{-1}\nabla^2_g\l+\frac1{(n-1)\l}g+\frac R{2(n-1)}g
\end{split}
\end{equation}
where we have used  \eqref{critical-e1-s3}.
 Moreover, the Weyl curvature tensor is zero and so the Riemannian curvature tensor of $ g $ equals the Kulkarni-Nomizu product of $ S $ and $ g$, which
 together with \eqref{Schouten-e2} shows
 \begin{equation}\label{Weyl-e1}
 \begin{split}
 R(X,Y,Z,U)&= \frac{   R +2\l^{-1}  }{(n-1)(n-2)} \lf[g(X,Z)g(Y,U)-g(X,U)g(Y,Z)\ri]
\\
&+\frac1{(n-2)\l}\bigg[  \nabla^2\l(X,Z)g(Y,U)+\nabla^2\l(Y,U)g(X,Z)\\
&-\nabla^2\l(X,U)g(Y,Z)
-\nabla^2\l(Y,Z)g(X,U)\bigg]
            \end{split}
\end{equation}
for all vector fields $X, Y, Z, U$. Here $R(X,Y,Z,U)$ is defined as $\la R(X,Y)U,Z\ra$ with $R(X,Y)=\nabla_X\nabla_Y-\nabla_Y\nabla_X-\nabla_{[X,Y]}$.

By \eqref{Schouten-e1} and \eqref{Schouten-e2}, we have:
\begin{equation}\label{Schouten-e3}
\begin{split}
0 = & D_X\lf(\l^{-1}\nabla^2_g\l\ri)(Y,Z)-D_Y\lf(\l^{-1}\nabla^2_g\l\ri)(X,Z) -\frac{
X(\l)g(Y,Z)-Y(\l)g(X,Z)}{(n-1)\l^2}\\
= & \l^{-1}\lf[D_X\lf(\nabla^2_g\l\ri)(Y,Z)-D_Y\lf(\nabla^2_g\l\ri)(X,Z)\ri]-\l^{-2}\lf(X(\l)\nabla^2\l(Y,Z)-Y(\l)\nabla^2\l (X,Z)\ri)\\
& -\frac{1}{(n-1)\l^2}
\lf(X(\l)g(Y,Z)-Y(\l)g(X,Z)\ri)\\
= & \l^{-1}R(X,Y,Z,\nabla\l)-\l^{-2}\lf(X(\l)\nabla^2\l(Y,Z)-Y(\l)\nabla^2\l (X,Z)\ri)\\
&-\frac{1}{(n-1)\l^2}
\lf(X(\l)g(Y,Z)-Y(\l)g(X,Z)\ri).
\end{split}
\end{equation}
Let $X$ be tangential to $N$ and let $Z=Y=\nabla\l$, we
have $$
0=Y(\l)\nabla^2\l(X,Z)=\frac12|\nabla\l|^2X(|\nabla\l|^2).
$$
on $N$. Hence $|\nabla\l|$ is constant on $N$. This proves
(i).

 In \eqref{Schouten-e3},
 let $X, Z$ be tangential to $N$, $Y=\nabla\l$ and
let $\xi=\nabla\l/|\nabla\l|$, we have
\begin{equation}\label{Schouten-e4}
   R(X,\xi,Z,\xi)=-\l^{-1}\lf(\nabla^2\l(X,Z)+\frac1{n-1}g(X,Z)\ri).
\end{equation}
On the other hand, let $Y=U=\nabla\l$ and $X$, $Z$
be tangential to $N$ in \eqref{Weyl-e1}, we  have:
\begin{equation}\label{Weyl-e2}
\begin{split}
  R(X,\xi, Z,\xi)= &
  \frac{1}{(n-2)} \l^{-1}
  \lf[ \nabla^2 \l (X, Z) + \nabla^2 \l(\xi, \xi) g(X, Z) \ri] \\
  &
  + \frac{1}{(n-1)(n-2)} ( 2 \l^{-1} + R ) g (X, Z).
 \end{split}
  \end{equation}
Comparing \eqref{Schouten-e4} and \eqref{Weyl-e2}, we have
\begin{equation}\label{SW-e1}
 \begin{split}
 (n-1)\nabla^2\l(X,Z) & = \lf[ - \nabla^2 \l(\xi, \xi)
 - \frac{n + R \l}{n-1} \ri] g(X, Z) \\
 & =  \lf[ - \nabla^2 \l(\xi, \xi)
+ \Delta_g \l \ri] g (X, Z)
 \end{split}
 \ee
 where
 the last step follows from
 \begin{equation}\label{critical-e2}
\Dg\l=-\frac1{n-1}\lf( R\l+n\ri)
\end{equation}
which is obtained by taking the trace of \eqref{critical-e1-s3}.
 Recall
 $$ \Delta_g \l = \Delta_g^N \l + H \frac{\partial \l}{\p \xi}
 + \nabla^2 \l (\xi, \xi) ,$$
 where $ H $ is the mean curvature of $ N$ and $\Dg^N$ is the Laplacian on $N$.
 Thus, \eqref{SW-e1} becomes
 \begin{equation}\label{umbilic-e1}
 \begin{split}
| \nabla \l |^{-1} \nabla^2\l(X,Z) & = \frac{H}{n-1}  g (X, Z) .
 \end{split}
 \ee
Now let $ A(X, Z) = g( \nabla_X \xi, Z)  $ be the second fundamental form of $N$, then
\begin{equation}\label{umbilic-e2}
A(X,Z)= \frac{\nabla^2\l(X,Z)}{|\nabla \l|}= \frac{H}{n-1}  g (X, Z),
\end{equation}
which shows $N$ is totally umbilical.

To prove that $H$ is constant on $N$. Let $ \alpha = H/(n-1)$.
By \eqref{Schouten-e3} and the
Codazzi-Mainardi equation for $X,Y, Z$ tangential to $N$,
we have
\begin{equation}
  \begin{split}
    0 &= R(X,Y,Z,\xi) \\
      & =\lf(\nabla_X^NA\ri)(Y,Z)-\lf(\nabla_Y^NA\ri)(X,Z)\\
      &=X( \alpha )g(Y,Z)-Y( \alpha )g(X,Z)
  \end{split}
\end{equation}
where $\nabla^N$ the covariant derivative of $N$. For
any given $X$,  let $Y=Z$ be a unit vector perpendicular to $X$.
Then $X( \alpha )=0$. Hence $\alpha$ is constant on $N$.
This proves (ii).

To prove (iii), let $X=Z$ and $Y=U$ in \eqref{Weyl-e1} and
choose $X$ and $Y$ to be orthonormal tangent vectors tangent
to $N$. It follows from \eqref{Weyl-e1}, \eqref{umbilic-e1}
and the fact $|\nabla\l|$ and $H$ are constant on $N$
that $R(X,Y,X,Y)$ is constant on $N$.
 By the Gauss equation and (ii), we
conclude that $N$ has constant sectional curvature.

\end{proof}

In the rest of this section, we assume that $ (\Omega, g) $
is a connected, compact Riemannian manifold
with a smooth (possibly disconnected) boundary $ \Sigma $.
Moreover, we make the following assumption on
$ (\Omega, g) $:

\vskip .2cm
{\sc Assumption A}: {\it $(\Omega^n,g)$ is conformally flat
and there is a smooth function $\l$ satisfying
\eqref{critical-e1-s3} and vanishing on $ \Sigma $.
Furthermore,  the first Dirichlet
eigenvalue of $(n-1)\Dg+R$ is nonnegative.}\vskip .2cm

Note that the condition on the first Dirichlet eigenvalues is automatically satisfied if $R\le 0$.

Given such an $ (\Omega, g ) $,
by \cite{MiaoTam08} we have $ \l > 0 $ in the interior of $ \Omega $.
In addition, if $ \nu $ denotes the outward
unit normal to $ \Sigma $,
then   $\frac{\p\l}{\p\nu}<0$ and is
constant on each connected component of $\Sigma$.
Similar to \cite{KobayashiObata1981}, we can now prove the
following result.

\begin{lma}\label{CF-l2} Let $\Sigma_0$ be a
connected component of $\Sigma$.
Let $\wt\Omega_0$ be the connected component of
the open set  $\{|\nabla\l|>0\} $ in $ \Omega $ such that its closure
contains $ \Sigma_0$ and let $\Omega_0=\wt\Omega_0\cup\Sigma_0$.
Then there exists a constant $ \delta_0 > 0 $ such that
$(\Omega_0,g)$ is isometric to
a warped product
$ ( [0, \delta_0) \times \Sigma_0,   ds^2+r^2h )$,
where $ r  > 0  $ is a smooth function on
$ [0, \delta_0)$ and $ h$ is the induced metric
on $\Sigma_0$ from $g$.
Moreover, $ \l $ on $ \Omega_0 $
depends only on $ s \in [0, \delta_0)$, and
$\Sigma_0$ has constant sectional curvature.
\end{lma}

\begin{proof} The claim that $\Sigma_0$ has constant
sectional curvature is a direct corollary of  Lemma \ref{CF-l1}
and the facts  $ \l=0 $ on $ \S_0 $, $ \frac{\p \l}{\p \nu} \neq 0 $ on
$ \S_0 $.
On $ \Omega_0 $, define the smooth vector field
$v=\nabla\l/|\nabla\l|^2$. This vector field is smooth up to $\Sigma_0$.
 For any $ x\in \Sigma_0$,
  let $\zeta_x(s)$ be the integral curve of $v$ such
that $\zeta_x(0)=x$. Then $\zeta_x$ can be extended until
it meets   the boundary of $\Omega_0$. Suppose   $\zeta_x$
is defined on $[0,\delta_x)$, then
\begin{equation}\label{metricstructure-e1}
\l(\zeta_x(s))=\int_0^s  g( \nabla\l(\zeta_z(\tau)),
\zeta_x^\prime (\tau) g d\tau+\l( x ) = s.
\end{equation}
Hence if $[0,\delta_x)$ is the maximal domain of the definition
of $\zeta_x$ and $ \max_{\Omega} \l $ is the maximum value of
$ \l $ on $ \Omega$, then $\delta_x \leq \max_{\Omega} \l < \infty$.  Note that $\l(\zeta_x(s))$ is increasing
in $s$ and $\frac{\p\l}{\p\nu}<0$ on $\Sigma$, it is easily seen
 that for any $s_i\to \delta_x$, $\zeta_x(s_i)$
cannot converge to a point at $\Sigma$.

 We claim that  $\delta_x$ is constant on $\Sigma_0$.
 It is sufficient to prove that $\delta_x= \delta$ where
  $\delta=\inf_{y\in\Sigma_0}\delta_y  $ which is positive
  as $ \S_0 $ is compact. Suppose $\delta_x > \delta$
  for some $ x \in \S_0 $, then $|\nabla\l |  \ge c >0$ on
 $\zeta_x(\delta-\e,\delta+\e)$ for some constants $c$ and
$\e>0$.
For any $s\in (0,\delta)$, let
$N_s=\{\zeta_y(s)|\ y\in \Sigma_0\}$.
Then $ | \nabla \l | > 0 $ on $ N_s $, and
$ \l = s $ on $ N_s $ by \eqref{metricstructure-e1}.
Moreover, $ N_s $ is connected as $ \S_0 $ is connected.
Therefore, Lemma \ref{CF-l1} implies that $ | \nabla \l | $
is constant on $ N_s $.  Consequently,  $|\nabla\l| \ge c$ on
$ N_s $ for all $ s \in (\delta-\e,\delta)$. This implies
that all $\zeta_y$ can be extended up to $\delta+\e'$ for
some $\e'>0$  independent of $y$, which contradicts the
definition of $\delta$. Hence $\delta_x=\delta$ for all
$x\in \Sigma_0$.

Let $I= [0,\delta)$ and define the map $\Phi : I \times
\Sigma_0 \rightarrow \Omega_0$ by $\Phi(s,x)=\zeta_x(s)$,
then $ \Phi $ is  an injective, local diffeomorphism. It is also true that
$\Phi(I\times \Sigma_0)$ is closed in  $\Omega_0$ because
if $x_k\in \Phi(I\times \Sigma_0)$ with $x_k\to x\in \Omega_0$, and if $x\notin \Phi(I\times \Sigma_0)$, then $\nabla\l(x)=0$, contradicting the definition of $\Omega_0$. Since $ \Omega_0 $ is connected,
we conclude $\Omega_0=\Phi(I\times \Sigma_0)$.

Let $(u_1,\dots,u_{n-1})$ be some
local coordinates on $\Sigma_0$,
then $\Phi_*(\p_s)=\vec v$ is
orthogonal to $\Phi_*(\p_{u_i})$.
Writing  $\Phi_*(\p_{u_i})$ as $ \p_{u_i}$, we have
\begin{equation*}
\begin{split}
\frac{\p}{\p s}g(\p_{u_i},  \p_{u_j})
&=g(\nabla_{\p_{u_i}} \vec v , \p_{u_j})+g( \p_{u_i}, \nabla_{\p_{u_j}}
\vec v)\\
&=2|\nabla\l|^{-1}\II( \p_{u_i},\p_{u_j})
  \end{split}
\end{equation*}
where $ \II $ is the second fundamental form of $ N_s $ with respect to $ \vec v$.
By Lemma \ref{CF-l1},
$
\II (\p_{u_i} , \p_{u_j} ) =  \alpha g( \p_{u_i},\p_{u_j})
$
for some function  $ \alpha $ depending only on $ s $,
moreover  $ | \nabla \l | $ also depends  only on $ s $.
Therefore,  in terms of coordinates $(s, u_1, \ldots, u_{n-1}) $ on
 $\Omega_0 $, the metric $ g $ can be written as
 $$ g = |\nabla\l|^{-2}ds^2+ \beta h $$
where $\beta $ is a function of $s$ and $h$ is the induced
metric on $\Sigma_0$. Rescaling $ s $ using the fact
that $|\nabla\l|$ depends only on $s$, we may re-write $ g $
 as  $g=ds^2+r^2 h$, where
$ s \in [0, \delta_0) $ for  some $ \delta_0 $
possibly different from $ \delta $, and  $ r $
is some function  depending only on $ s $.
The fact $ \delta_0 < + \infty $ follows from the assumption that
$ \Omega $ is compact.
 \end{proof}

Let $\Sigma_0$, $I=[0,\delta_0)$, $\Omega_0$,
$ r $ and $ h $ be  given as in Lemma \ref{CF-l2}. We identify
$I\times \Sigma_0$ with $\Omega_0$ using the isometry.
Since $\frac{\p \l}{\p\nu}<0$ on $\S_0 $ and
$ | \nabla \l | > 0 $ on $ I \times \S_0 $,
we have  $\l^\prime (s)  >0$ on $I \times \S_0 $, where
`` $'$ " denotes the derivative w.r.t $s$.
For convenience, we also normalize $ R $ so that
$R=n(n-1) \kappa $ with $ \kappa = 0, 1$ or $-1$.
By Proposition 1.1 in section 2, we have
\begin{equation}\label{warped-e1}
r''+\kappa r=ar^{1-n}
\end{equation}
for some constant $a$,
and
\begin{equation}\label{warped-e3}
\frac{r'}r\l'-\frac{r''}r \l=-\frac1{n-1}.
\end{equation}
Also from section 2, we have
\begin{equation}\label{Rics-eq}
\Ric(g)(\p_s,\p_s)=-(n-1)\frac{r''}r= (n-1) \kappa - (n-1) ar^{-n}
\end{equation}
and
\begin{equation} \label{chieq}
 ( r^\prime )^2 + \kappa r^2 + \frac{2a}{n-2} r^{2-n} = \kappa_0
 \end{equation}
 where $\kappa_0$ is the sectional curvature of $(\Sigma_0,h)$
 which is a constant.

In what follows, we let $ \ol{ I \times \Sigma_0 } $ be
the closure of $ I \times \Sigma_0 $
in $ \Omega\cup\Sigma_0 $. Since $\nabla\l=0$ somewhere in $\Omega$,   $ \ol{ I \times \Sigma_0 } \setminus { I \times \Sigma_0 }$
is not empty and consists of points at which $\nabla\l=0$.

\begin{lma}\label{CF-l3} With the above notations, the following are true:

\begin{enumerate}
  \item [(i)] $a\ge0$.
  \item [(ii)] If $a=0$, then $(\Omega,g)$ is a geodesic ball in space forms.
  \item [(iii)] If $a>0$, then $S_0=\ol{I\times\Sigma_0}\setminus  I\times\Sigma_0$ is a connected, embedded, totally
   umbilical hypersurface with constant mean curvature in
   $\Omega$.
   Moreover, for any $p\in S_0$ there is an open neighborhood $U$
of $p$ such that $U\cap S_0=U\cap S$, where $S=\{q\in
\Omega|\ \nabla\l(q)=0\}$.
\end{enumerate}
\end{lma}
 \begin{proof} (i) Suppose $a<0$. By \eqref{Rics-eq}, we
 have $\liminf_{s\nearrow\delta_0}r(s)>0$. Suppose there exists
 $s_k\nearrow\delta_0$ such that  $r(s_k)\to\infty$, then \eqref{chieq} implies
 $$
 \kappa+ \lf( \frac{ r'(s_k) }{ r(s_k) } \ri)^2 \to 0 .
 $$
 In particular, $\lf(\frac{r'(s_k)}{r(s_k)}\ri)^2 $ are uniformly bounded.
 Since $\l'(s_k)\to 0$, by \eqref{warped-e1} and \eqref{warped-e3}
 we have:
 $$
 -\kappa \lim_{k \to \infty} \l( s_k )=\frac{1}{n-1},
 $$
which is impossible if $\kappa=0, 1$ because $ \l > 0 $ in the interior of $ \Omega $. Suppose $\kappa=-1$. By \eqref{warped-e1}, $r^{\prime \prime} <  r$. Let $f$ be a function on $ I $
 such that $f''=f$, $f(0)=r(0)$ and $f'(0)>r'(0)$. Then $f>r$ near 0. Since $f$ is bounded on $I$ and $r(s_k)\to\infty$, there exists $s_0>0$ such that $f>r$ on $(0,s_0)$ and $f(s_0)=r(s_0)$. So we have $(r-f)^{\prime \prime} <  (r-f)$, $r-f<0$ on $(0,s_0)$, but $r-f=0$ at $0, s_0$. This is impossible.
Hence, we have $ \limsup_{s \nearrow \delta_0}  r(s) < + \infty $.
 It follows that $C^{-1}\le r\le C$ on $I$ for some $C>0$. In particular, $r$ can be extended smoothly beyond $\delta_0$ satisfying \eqref{warped-e1},
 and  $ \l $ can be extended smoothly beyond $\delta_0$ satisfying \eqref{warped-e3}.
 At $ \delta_0 $, we have $\l'(\delta_0)=0$, hence  \eqref{warped-e1} and \eqref{warped-e3} imply
 \be \label{limit-eq1}
 \l(\delta_0)=\frac{1}{n-1}\cdot\frac{1}{-\kappa+ar^{-n}(\delta_0)}.
 \ee
 Again this is impossible if $\kappa=0, 1$. Suppose $\kappa=-1$, then
 $ \l(\delta_0)>\frac1{n-1}$.
Recall that $ \l =0 $ on $ \Sigma $ and by \eqref{critical-e2} we have
\be \label{critical-e2-n}
\Dg \lf(\l-\frac{1}{n-1}\ri)-n \lf(\l-\frac{1}{n-1}\ri)=0  \ \ \mathrm{on} \ \Omega .
\ee
 Hence, $\max_\Omega\l\le \frac{1}{n-1}$, contradicting
$ \l(\delta_0)>\frac1{n-1}$. This proves (i).

(ii) Suppose $a=0$. Then \eqref{warped-e1} becomes $r''+\kappa r=0$. Hence $r$ can be defined for all $s$. In particular, $\lim_{s\to\delta} r(s)=r_0$  exists.  Suppose  $r_0>0$, then $\l$ can be
extended beyond $\delta$ satisfying \eqref{warped-e3}.
As in case (i), it follows from \eqref{warped-e1} and \eqref{warped-e3} that
 $$
- \kappa  \l(\delta_0)=\frac{1}{n-1} ,
 $$
which is impossible if  $\kappa=0$ or $ 1 $. If $\kappa=-1$,
then $\l(\delta_0)=\frac1{n-1}$. However, by the proof of (i), we have
 $\max_\Omega\l\le \frac1{n-1}$. Thus, the  function
 $ \l - \frac{1}{n-1} $, which is not a constant,  achieves an interior maximum which is zero.
By \eqref{critical-e2-n}, we get a contradiction to the
strong maximum principle. Therefore,  $ r_0 = 0$. Consequently,
 $\ol{I\times \Sigma_0}\setminus I\times \Sigma_0$ consists of only one point, say $ p $. Let $B_p \subset \Omega $ be  a connected, open
 neighborhood  of $ p $. Then
 $ \lf(B_p \setminus\{p\} \ri) \cap \lf( I\times \Sigma_0 \ri) $ and
$ \lf(B_p \setminus \{p\}\ri) \setminus \ol{ I\times \Sigma_0 } $
 are both open sets in $B_p \setminus\{p\}$. Hence
 $ \lf(B_p \setminus \{p\}\ri) \setminus \ol{ I\times \Sigma_0 } = \emptyset$.
 As $ \Omega $ is connected,  we conclude that $\Omega= \ol{I\times \Sigma_0}$. As $ a = 0$, by Remark \ref{s2-r1}  the metric $g$ is Einstein. Hence $(\Omega,g)$ is a geodesic ball in space forms by Theorem \ref{Einstein-t1}.

 (iii) Suppose $a>0$. By Lemma \ref{lma-kb-s2}, $r$ can be extended to be a solution on $\R$ and is bounded below away from zero. Hence $ r $ satisfies  $ C^{-1} \leq r \leq C $ in $[0,\delta_0)$ for some positive  constant $ C $.

For each $ y \in \Sigma_0 $, let $ \alpha_y (s) $ denote the geodesic  starting from $ y  $ with $ \alpha_y^\prime (0) = \p_s $.
As $ \Sigma_0 $ is compact,  there exists $ \delta_1 > \delta_0 $
such that $ \alpha_y (s) $ is defined on $[0, \delta_1)$ for all
$ y \in \Sigma_0 $.  Clearly,
the set $ \{ \alpha_y (\delta_0) \ | \ y \in \Sigma_0 \} $ is contained
in  $ S_0 $. On the other hand, for any $p\in S_0$, there exists $ (s_k,x_k) \in I \times \Sigma_0 $  such that $ \alpha_{x_k}(s_k)   \to p$ with $s_k \to \delta_0$.
As $\Sigma_0$ is compact and $r \leq C $, there exists
$ x \in \Sigma_0 $ such that $ \alpha_{x} (s_k) \to p $.
 Hence, $ p = \alpha_{x} (\delta_0) $. This shows
$ S_0 = \{ \alpha_y (\delta_0) \ | \ y \in \Sigma_0 \} $,
in particular $ S_0 $ is connected (as $ \Sigma_0 $ is connected).

To show $ S_0 $ is an embedded hypersurface, we
let $\Sigma_s=\{s\}\times \Sigma_0$ for each $0<s<\delta_0$.
Since  the induced metric on $\Sigma_s$ is $r^2h$ and $ r \geq C^{-1} $, the curvature of $\Sigma_s$ is bounded by a constant independent of $s$.   Since $\Sigma_s$ is totally umbilical, it follows from the Gauss equation   that the norm of the second fundamental form of $\Sigma_s$ is  also bounded by a constant independent of $s$.
 For any $ p = \alpha_x (\delta_0) \in S_0 $,  using the estimates   Lemma \ref{graph1} in the Appendix  and the fact that $\Sigma_s$ is of constant mean curvature  with uniformly bounded second fundamental form, we  conclude that  there exists $\rho>0$  and
 a sequence $ s_k \nearrow \delta_0 $  such that
$N_k=\{(s_k,y)|\ y\in B_0(x,\rho)\}$ converges to an
embedded hypersurface $N_p$ passing through $p$. Here
$B_0(x,\rho)$ denotes a geodesic ball in $(\Sigma_0, h)$
centered at $ x $ with radius $ \rho$. Note that $N_p\subset S_0$.

 At $ p = \alpha_x (\delta_0) $, we have
$ \nabla \l = 0 $. By  \eqref{critical-e1-s3} and \eqref{Rics-eq},
we have
\be \label{Ricatp}
\nabla^2_g \l ( \alpha_x^\prime( \delta_0), \alpha_x^\prime (\delta_0) ) =
- \kappa \l - \frac{1}{n-1} - ( n-1) a r^{-n} .
\ee
As $ a > 0 $, \eqref{Ricatp} implies
$ \nabla^2_g \l ( \alpha_x^\prime( \delta_0), \alpha_x^\prime (\delta_0) ) < 0 $.
 This is obvious if
$\kappa=0$ or 1. If $\kappa=-1$, this follows from the fact
$ \max_\Omega \l<\frac1{n-1}$.

By shrinking $B_0(x,\rho)$ if necessary,  we may assume that
there exist small positive constants $b $ and $ c$ such that
$ \nabla^2_g \l ( \alpha_y^\prime( s), \alpha_y^\prime (s) ) < - c $
 for all $y\in B_0(x,\rho)$ and $ s \in (\delta_0-b,  \delta_0+b)$.
 As $ \nabla \l = 0$ at $ \alpha_y (\delta_0)  \in S_0$,
 we have   $ g (\nabla \l  ,  \alpha_y^\prime ( s) ) \neq 0$ for all
$y\in B_0(x,\rho)$ and $ s \in (\delta_0-b,  \delta_0+b)$.
In particular, $\nabla \l$ is not zero at
the points $\alpha_y(s)$ for all such $y$ and $s$.

We want to show that there is an open neighborhood $U$ of $p$
such that $U\cap S = U\cap N_p$. If not, then
 there exist a sequence of points $ \{ p_k \} \subset \Omega$
 such that $p_k\to p$, $p_k\notin N_p$ and
 $ p_k \in S $. For $ k $ sufficiently large, there exists minimal geodesics $\beta_{p_k} ( t ) $  starting from $ p_k $ and ending at
$ N_p$  so that $\beta_{p_k}^\prime (t) $ is perpendicular to $N_p$ at
some point $q_k= \alpha_{y_k} (\delta_0) $ for some
$y_k\in B_0(x, \rho)$. Since $ \beta_{p_k} $ and $ \alpha_{y_k} $
are two geodesics both perpendicular to $ N_p $ at
$ q_k $, we must have $ p_k = \alpha_{y_k} ( s_k ) $
for some $ s_k $. Moreover, $ s_k \neq \delta_0 $
as $ p_k \notin N_k $. When $ k $ is sufficiently large,
we have $ s_k \in (\delta_0 - b , \delta_0 + b) $, hence
$\nabla \l$ is not zero at $ p_k = \alpha_{y_k} ( s_k)$,
contradicting the fact that $ p_k \in S $.

As $ S_0 \subset S $, we conclude that  $S_0$ is an
embedded hypersurface in $ \Omega $ such
that for, each $ p \in S_0 $, there is an open neighborhood $ U $
of $ p$ such that  $U\cap N_p = U \cap S$.
The fact that $S_0$ is  totally
umbilical and has constant mean curvature follows directly from the
fact that each $\Sigma_s$ is totally umbilical and has
constant mean curvature.

 \end{proof}

 Let $(\Omega,g)$ be  given as before. Assume that
 $(\Omega,g)$ is not a geodesic ball in space forms.
 Let  $\Sigma_1,\dots,\Sigma_k$ be the connected components
 of the boundary $ \Sigma$.  For each $ i = 1, \ldots, k $,
 let $\Omega_i$ be  the connected component of the open set $\{|\nabla\l|> 0\}$   in $ \Omega $ whose closure contains $ \Sigma_i $.
 By Lemma \ref{CF-l2}, each $\Omega_i$  can be then
 identified with  $I_i\times\Sigma_i$ where
 $ I_i =  (0,\delta_i)$ for some $ 0 < \delta_i < \infty $.
 On $ I_i \times \Sigma_i $, the metric $ g $
has the form $ds^2+r_i^2h_i$, where $ r_i $ is some
smooth positive function on $ I_i $.
By \eqref{warped-e1} and Lemma \ref{CF-l3}, each
$r_i$ satisfies $r_i''+\kappa r=a_ir^{1-n}$ for
some constant $a_i>0$.
Let $ \ol{I_i\times\Sigma_i} $ be the closure of $ I_i \times \Sigma_i $
in $ \Omega $ and let
$S_i=\ol{I_i\times\Sigma_i}\setminus I_i\times\Sigma_i$.
By Lemma \ref{CF-l3}, each $ S_i $ is a connected, embedded,
totally umbilical hypersurface with constant mean curvature
in the interior of  $ \Omega$.

\begin{lma}\label{CF-l4}
With the above assumptions and notations, $ \Sigma $ at most
has two connected components, i.e. $ k \leq 2$.
If $ k = 2 $, then $ S_1 = S_2 $ and
$ \Omega = \ol{ I_1 \times \Sigma_1}
\cup \ol{ l_2 \times \Sigma_2} $.
If $ k = 1 $, then $ \Omega = \ol{ I_1 \times \Sigma_1}$.
\end{lma}

\begin{proof}
Suppose $ k \geq 2 $ and suppose $ S_i \cap S_j \neq \emptyset $
for some $ i \neq j $, say $ i =1, j=2 $.  For any $ p \in S_1 \cap S_2$, Lemma \ref{CF-l3} implies
there exists an open neighborhood $ U $ of $ p $ in $ \Omega $
such that $ U \cap S_1 = U \cap S $, where $ S = \{ \nabla \l = 0 \} $.
As $ S_2 \subset S $, we have $ U \cap S_2 \subset U \cap S_1 $.
As $ S_1 $ and $ S_2 $ are embedded hypersurfaces, the above
implies $ S_1 \cap S_2 $ is an open subset of both  $ S_1 $ and $ S_2 $. As $ S_1 $ and $ S_2 $ are connected,
we have $ S_1 = S_1 \cap S_2 = S_2 $.
Now, every geodesic in $\Omega$ emanating from
and perpendicular to  $ S_1 = S_2 $ is either contained
in $ I_1 \times \Sigma_1 $ or $ I_2 \times \Sigma_2 $. Hence,
$ \ol{ I_1 \times \Sigma_1 } \cup \ol{ I_2 \times \Sigma_2 } $
is both an open and a closed set in $ \Omega$. As $ \Omega $
is connected, we must have $ \Omega =  \ol{ I_1 \times \Sigma_1 } \cup \ol{ I_2 \times \Sigma_2 } $ and $ k = 2 $.

Suppose $ k \geq 2 $ and suppose $ S_i \cap S_j = \emptyset$  for  any $ i \neq j $. We prove that this is impossible by
considering
$ U =  \Omega \setminus \cup_i \ol{ I_i \times \Sigma_i }$.
If $ U = \emptyset$, then each $ \ol{ I_i \times \Sigma_i }$
would be both open and closed in $ \Omega$,
contradicting the fact that $ \Omega $ is connected.
Suppose $ U \neq \emptyset $. If $ \kappa = 0 $ or $ 1 $,
then \eqref{critical-e2} implies
$$ \Delta_g \l = - n \kappa \l - \frac{n}{n-1} < 0 , $$
where we also used $ \l  > 0 $ in the interior of $ \Omega$.
Hence, $ \min_{\bar{U}} \l $ could only occur at $ \p U = \cup_i S_i $.
Suppose $ p \in \p U  $ such that $ \l (p) =  \min_{\bar{U}} \l $,
then the strong maximum principle  implies
$ \frac{\p \l}{\p \nu_U}  \neq 0 $, where $ \nu_U $ is a unit normal
vector to $ \p U $ at $ p$. This contradicts to
the fact that $ \nabla \l = 0 $ at points in $ S_i $.
If $ \kappa = - 1$, then as in
the proof of Lemma \ref{CF-l3}, we have $ \l < \frac{1}{n-1} $
on $ \Omega $ and
$$
\Delta_g \lf( \l - \frac{1}{n-1} \ri) = n \lf( \l - \frac{1}{n-1} \ri) .
$$
Applying the strong maximum principle to $ \l - \frac{1}{n-1} $
on $ U$, we get a contradiction as before.
Therefore, we conclude that if $ k \geq 2 $, then
$ k = 2 $, $ S_1 = S_2 $ and $ \Omega = \ol{ I_1 \times \Sigma_1}
\cup \ol{ l_2 \times \Sigma_2} $.

Next, suppose $k = 1$. Let $ U = \Omega \setminus
\ol{ I_1 \times \Sigma_1} $. The exact same argument
in the previous paragraph implies $ U = \emptyset$.
We conclude $ \Omega = \ol{ I_1 \times \Sigma_1}$.

  \end{proof}

\begin{thm}\label{CF-t1}
Let $ (\Omega, g) $ be a connected, compact Riemannian manifold
with a disconnected boundary $ \Sigma $. Suppose $(\Omega, g)$  satisfies ASSUMPTION A.
Then  $(\Omega,g)$ is one of  the manifolds constructed in Example {\bf (1)} after Proposition \ref{examples-p1}.
  \end{thm}

\begin{proof}
By Lemma \ref{CF-l4}, $\Sigma $ has exactly two
connected components, say $ \Sigma_1 $ and $ \Sigma_2$.
Moreover, if $ I_i = [0, \delta_i)$,
$ I_i \times \Sigma_i $, $ \ol{I_i \times \Sigma_i }$
and $ S_i $ are given as in Lemma \ref{CF-l4} for $ i =1$,  $2$.
then $ S_1 = S_2 $ and
$ \Omega = \ol{ I_1 \times \Omega_1}
\cup \ol{ I_2 \times \Omega_2} $.
 On each $I_i \times \Sigma_i $,
 by Lemma \ref{CF-l2} and \ref{CF-l3},
the metric $ g $ has the form
\be \label{geq}
 g = d s^2 + r^2_i h_i
\ee
 where $ h_i $ is a metric on $ \Sigma_i$
with constant sectional curvature  and $ r_i $ is
a smooth positive function on $ I_i $ satisfying
\be \label{req}
r_i^{\prime \prime} + \kappa r_i = a_i r_i^{1-n}
\ee
for some constant $ a_i > 0 $. Here we normalized $ g $
so that its scalar curvature, which is a constant, is
$ n (n-1) \kappa $ with $ \kappa = 0$ or $ \pm 1 $.
Note that \eqref{geq} and \eqref{req} are invariant if
the triple $(r_i, h_i , a_i )$ is replaced by
$( c r_i, c^{-2} h_i, c^n a)$ for any $ c > 0 $.
Hence, after rescaling, we may assume that $ a_1 = a_2 $.

 Let $ S $ denote $ S_1 = S_2 $. For any $ p \in S$,
 there exists $ x  \in \Sigma_1 $, $ y \in \Sigma_2 $ such that
 $ \alpha_{x} ( \delta_1 ) = p = \beta_y (\delta_2) $.
 Here $ \alpha_{x}(s) $, $ \beta_{y}(s) $ denote the geodesic
 staring from $ x $, $ y$  with $ \alpha_{x}^\prime(0) = \p_s $,
 $  \beta_{y}^\prime (0) = \p_s $ respectively.
 At $ p$, $ \alpha_x^\prime ( \delta_1) $ and $ \beta_y^\prime (\delta_2)$ are both perpendicular to $ S$, hence
 $  \alpha_x^\prime ( \delta_1) = -  \beta_y^\prime (\delta_2)$,
 and $ \gamma = \alpha_x \cup ( - \beta_y ) $ is a geodesic
 in $ \Omega$,
 where $ - \beta_y (s)  $ is defined as $ \beta_y (-s) $.
 At $ p$, recall that  $ \nabla \l = 0 $. By \eqref{warped-e3},
 we  then have
 \be \label{rdelta}
 \frac{r_1^{\prime \prime} (\delta_1) }{r_1 (\delta_1) } =
  \frac{r_2^{\prime \prime} (\delta_2) }{r_2 (\delta_2) } .
 \ee
 It follows from \eqref{req}, \eqref{rdelta} and the fact $ a_1 = a_2 $
 that $ r_1 ( \delta_1) = r_2 ( \delta_2) $.
 On the other hand, using the fact that the mean curvature
 of $ S $ w.r.t $ \alpha_x^\prime (\delta_1) $ is negative
 of the mean curvature of $ S $ w.r.t
 $ \beta_y^\prime (\delta_2) $, we conclude
 $ r_1^\prime ( \delta_1) =  - r_2^\prime ( \delta_2)$.
 In particular, if we let $ I = [0, \delta_1 + \delta_2 ]$
 and define   $ r(s) = r_1 (s) , s \in [0, \delta_1] $
 and $ r(s) = r_2 (  \delta_1 + \delta_2 -s ) , s \in [ \delta_1, \delta_1 + \delta_2] $,
 then $ r (s) $ is a smooth function on $ I $ satisfying
 \be \label{geq-2}
 r^{\prime \prime} + \kappa r = a r^{1-n},
 \ee
  where $ a  = a_1 = a_2 $ is some positive constant.

 Now suppose there exists  another $ \tilde{x} \in \Sigma_1 $
  such that $ \alpha_{\tilde{x}} (\delta_1) = p $, then we would
  have
  $   \alpha_{\tilde{x}}^\prime ( \delta_1) =  \alpha_{x}^\prime ( \delta_1)  $,  hence $ \tilde{x} = x $. This implies the maps
  $ x \mapsto \alpha_x (\delta_1) $, $ y \mapsto \beta_y ( \delta_2) $
  are bijective from $ \Sigma_1$, $ \Sigma_2 $ to $ S$.
  Consequently, the map $ (x, s) \mapsto \gamma_x(s) $,
  where $ \gamma_x (s) $ is the geodesic staring from $ x \in \Sigma_1 $ and perpendicular to $ \Sigma$, is a diffeomorphism
  from $ I \times \Sigma_1 $ to $ \Omega$.
  By \eqref{geq},  the induced metric from $ g $ on $ S = \{ \delta_1 \} \times \Sigma_1 $  is given by  both
  $ r_1^2 (\delta_1) h_1 $ and $ r_2^2 ( \delta_2) h_2 $.
  As  $ r_1 ( \delta_1) = r_2 ( \delta_2)$, we have $ h_1  = h_2 $.

 Note that $ r^\prime (0) = r^\prime_1(0) < 0 $
 and $ r^\prime ( \delta_1 + \delta_2) = - r^\prime_2 ( 0 ) > 0 $,
 hence there exists an $ s_0 \in I $ such that $ r^\prime (s_0) = 0 $.
 Replacing $ s $ by $ s - s_0$,   we conclude that
 $ \Omega $ is isometric to $ I  \times \Sigma_1$, where
 $ I $ is replaced by $ (- s_0, \delta_1 + \delta_2 - s_0 )$ and the
 metric $ g $ is given by
 $ g = d s^2 + r^2 h $ with  $ r $ satisfying \eqref{geq-2}
 and $ r^\prime (0) = 0 $. Moreover, $ \l $ only depends on $ s \in I $.
 As $ a > 0 $, by Section 2 we know both $ r $ and $ \l $ can be
 extended to $ \R^1 $ and $ 0 $. Therefore, $ (\Omega, g)$
 is one of the examples in Example {\bf (1)} after Proposition \ref{examples-p1}.
 \end{proof}

  \begin{thm}\label{CF-t2}
  Let $ (\Omega, g) $ be a connected, compact Riemannian manifold
with a connected boundary $ \Sigma $. Suppose $(\Omega, g)$  satisfies ASSUMPTION A.  Then  $(\Omega,g)$ is either
a geodesic ball in a simply connected space form or  one of the manifolds constructed in   in Example {\bf (2)} after Proposition \ref{examples-p1}.  \end{thm}

\begin{proof} Suppose that $(\Omega,g)$ is not
a geodesic ball in a simply connected space form.
Since the boundary $ \Sigma $ is connected,
by Lemmas \ref{CF-l2}, \ref{CF-l3} and \ref{CF-l4},
we have
\ $\Omega=\ol{I\times\Sigma}$ (closure is taken with respect to $\Omega$) with the metric $ g $ on $ I \times \Sigma $ given by
$ds^2+r^2(s)h$, where $I=[0,\delta)$ for some positive number $\delta$. Now the functions $r$ and $\l$ satisfy \eqref{warped-e1} 
(with $a>0$) and \eqref{warped-e3}.  
Moreover, $S=\ol{I\times\Sigma}\setminus  I\times\Sigma$ is
a connected, embedded  hypersurface
in the interior of $ \Omega$.
Let  $(U; x^1, \ldots, x^n)$ be a
 local coordinate in $ \Omega $ such that
$ U \cap S = \{ x^n = 0 \} . $
Let $ U_+ = \{ x \in U \ | \ x^n > 0 \}$ and
$ U_- = \{ x \in U \ | \ x^n < 0 \} $.
Since $ \Omega \setminus
S = I \times \Sigma_0 $, both $ U_+ $ and $U_- $
are contained in $ I \times \Sigma_0 $. In particular,
as $ s \nearrow \delta $, the surfaces
$ (\{ s \} \times \Sigma_0 ) \cap U_+ \ \mathrm{and} \
( \{ s \} \times \Sigma_0 ) \cap U_- $
converges to $ S \cap U $ from two sides of $ S \cap U $ in
 $U $. As the mean curvature $ H_s $ of
$ \{ s \} \times \Sigma_0 $ is constant for each $ s \in I$,
the mean curvature  $ H $ of $ S \cap U $ is  given by
both $ \lim_{s \rightarrow \delta_-} H_s  $ and
$ - \lim_{s \rightarrow \delta_-} H_s  $. Hence, $ H = 0 $. Since $S$ is totally embilical with constant mean curvature by Lemma \ref{CF-l3},
we conclude that  $S$ is totally geodesic.

Now consider $ \tilde{M} = [ 0, \delta]\times \Sigma $
with the metric $ \tilde{g} = d s^2 + r^2 h $
(the fact $ a > 0 $ implies that $ r $ is smooth up to $ \delta $
with $ r (\delta ) > 0 $).
Let $ D \tilde{M} $ denote the doubling of $ (\tilde{M}, \tilde{g} )$
with respect to $ \Sigma_\delta =  \{ \delta \} \times \Sigma $,
which is totally geodesic
in $ \tilde{M} $. Then $ D \tilde{M} $ is one of the manifolds constructed in   Example (1) after Proposition \ref{examples-p1} 
 (with a reflection symmetry across a totally geodesic
hypersurface).
Let $ \Sigma_\delta $, $ S$ be equipped with the induced metric
from $ \tilde{g} $, $ g$. Consider the map
$ \phi : \Sigma_\delta \rightarrow S $ given by
$ \phi ( \delta, x) = \alpha_x(\delta) $,
where  $ \alpha_x (s) $ is the geodesic
in $ \Omega $ starting from $ x $ and perpendicular to $ \Sigma$.
It follows from the facts that $ S $ is an embedded hypersurface
and  each $ \alpha_x (s) $ is perpendicular to $ S $ at
$ \alpha_x (\delta) $ that $ \phi $ is a local isometry between
$ \Sigma_\delta $ and $ S $. Since $ \Sigma_\delta$ and $ S $
are both  compact, $ \phi $ is an covering map.
Let $ p \in S $, suppose there are three  points $ x $, $y$,
$z$ in $ \Sigma$ such that $ \alpha_x (\delta) = \alpha_y (\delta)
= \alpha_z ( \delta) = p $, then two of
$ \alpha_x^\prime (\delta)$, $ \alpha_y^\prime (\delta)$
$ \alpha_z^\prime (\delta)$ must be the same as all of them
are perpendicular to $ S $ at $ p $. Hence, $x$, $y$ and $z $
can not be distinct. This implies that
$ \phi $ is either injective or a double cover. If $ \phi $ is injective,
then the map $ x \mapsto \alpha_x(\delta)$ would be a diffeomorphism
from $ \Sigma $ to $ S$, hence $ \Omega \setminus \ol{I \times \Sigma} \neq \emptyset$, which is a contradiction. We conclude that
$ \phi $ is a double cover. Hence $(\Omega,g)$ is one of the manifolds constructed in Example {\bf (2)}  after Proposition \ref{examples-p1}.  
\end{proof}

Theorem \ref{CF-s0-t2}  in the introduction now follows directly from 
Theorem \ref{CF-t1}, Theorem \ref{CF-t2} and Proposition
\ref{prop-1-s2}.
Since the manifolds constructed in   Example {\bf (2)} after Proposition \ref{examples-p1}  are not simply connected, we also have the following:

\begin{cor} Let  $(\Omega,g)$ be as in Theorem \ref{CF-t2} satisfying ASSUMPTION A. Suppose $\Omega$ is simply connected, then $(\Omega,g)$ is a geodesic ball in a simply connected space form.
\end{cor}

\vskip .2cm

 \centerline{\appendix{APPENDIX}}\vskip .2cm

In this appendix, we  include estimates of graphical 
representation  of hypersurfaces with bounded
second fundamental form, which is needed in Section 4.

 Let $M^{n}$ be a complete Riemannian manifold and $N$ be an
 immersed hypersurface in $M$. Assume the following:
 \begin{itemize}
   \item [({\bf a1})] The curvature $Rm$ and the covariant
   derivative $DRm$ of the curvature of $M$ are bounded.
   \item [({\bf a2})] The injectivity radius of $M$ is bounded from
   below.
   \item[({\bf a3})] The norm of the second fundamental form
   of $N$ is uniformly bounded.
 \end{itemize}
 
 The following lemma was proved in \cite{Hamilton95-1}.
 
 \begin{lma}\label{normalcoordinates1} There exist $r>0$
 and constant $C>0$ depending only on the bounds in ({\bf
 a1}-{\bf a2}) such that for any $p\in M$, the exponential
 map is a diffeomorphism in $B(p,r)$. Moreover, if
 $(x^1,\dots,x^{n})$ are normal coordinates, then the
 component of the metric tensor in these coordinates
 satisfy:
 \begin{equation}\label{metric1}
   |g_{ab}-\delta_{ab}|(x)+|\frac{\p}{\p x^c}g_{ab}|(x)\le
   C|x|.
 \end{equation}
 \end{lma}
 
Next, we let $a,b,c,\dots$ denote indices ranging from $1$ to $n$ and
$i,j,k,\dots$ denote indices ranging from $1$ to $n-1$.
Let $p\in N$ and let $x^a$ be normal coordinates in
$B(p,r)$ such that $x^{n}=0$ is the tangent plane of $N$ at
$p$.
\begin{lma}\label{graph1} There exists $r>\rho_0>0$
independent of $p$ and a function
$w=w(x^1,\dots,x^{n-1})=w(x')$ defined on $|x'|\le \rho_0$,
where $|x'|^2=\sum_{i=1}^{n-1} (x^i)^2$, such that
$\{(x',w(x'))|\ |x'|\le \rho_0\}$ is part of $N$ passing
through $p$. Moreover, there is a constant $C_1$
independent of $p$ such that $|w|+|\p_i w|+|\p_i\p_j w|\le
C_1$ in $|x'|\le \rho_0$. Here $\p_i w=\frac{\p w}{\p x^i}$
and $\p_i\p_j w=\frac{\p^2 w}{\p x^i\p x^j}$.
 \end{lma}
 \begin{proof} Let $w$ be a function defined near $x'=0$
 such that $(x',w)$ part of $N$ near $p$ and $(x',w)$ is  inside $B(p,r)$. Suppose $\rho>0$
 is such that  the function $w(x')$ can be extended and defined
 in $|x'|\le \rho<r$ with $(x',w)$ being  part of
 $N$ and   inside $B(p,r)$. We want to estimate $|\p_iw|$.
 Let $W(x)=w-x^{n}$. Then the norm of the second
 fundamental form is:
 \begin{equation}\label{2ndfundamental1}
  \begin{split} |A|^2&=\sum_{1\le
   a,b,c,d\le n+1}\lf(g^{ac}-\frac{W^aW^c}{|DW|^2}\ri)
   \lf(g^{bd}-\frac{W^bW^d}{|DW|^2}\ri)\lf(\frac{W_{ab}}{|DW|}\ri)
   \lf(\frac{W_{cd}}{|DW|}\ri)\\
   &=\sum_{1\le
   a,b,c,d\le n+1}g^{ac}g^{bd}\lf(\frac{W_{ab}}{|DW|}\ri)
   \lf(\frac{W_{cd}}{|DW|}\ri)\\
   &-2\sum_{1\le
   a,b,c,d\le n+1}g^{ac}\frac{W^bW^d}{|DW|^2}\lf(\frac{W_{ab}}{|DW|}\ri)
   \lf(\frac{W_{cd}}{|DW|}\ri)\\
   &+\lf(\sum_{1\le a,b\le
   n+1}\frac{W^aW^bW_{ab}}{|DW|}\ri)^2\\
   &=I-2II+III\\
   &\ge I-2II
   \end{split}
 \end{equation}
where $DW$ is the gradient of $W$, $DW=W^a\frac{\p}{\p x^a}$, and
$W_{ab}$ is the Hessian of $W$. (See \cite{{SchoenYau81}}).  In the following, $C$ always denotes a constant depending
only on the bound in assumptions ({\bf a1}-{\bf a3}) and $n$, and
$f(\rho)$ is a function such that $|f(\rho)|\le C\rho$. They may
vary from line to line.

 Let $G(\rho)=\sup_{|x'|\le \rho}|\p w|$,
where $\p w= (\p_1 w ,\dots,\p_nw)$ and the norm is w.r.t. the
Euclidean metric. We have the following estimates for $|x'|\le
\rho$:
\begin{equation}\label{estimates1}
    |w(x')|  \le   G(\rho)|x'|
\end{equation}
\begin{equation}\label{estimates2}
 \begin{split}
   W^a&=g^{ab}W_b \\
      & =W_a+(g^{ab}-\delta^{ab})W_b\\
      &=W_a+\lf(1+G(\rho)\ri)f(\rho)
  \end{split}
\end{equation}
where $W_a=\frac{\p W}{\p x^a}$,
\begin{equation}\label{estimates3}
 \begin{split}
   W_{ij}&=\frac{\p^2 W}{\p x^i\p x^j}-\Gamma_{ij}^aW_a\\
      & =w_{ij}+ \lf(1+G(\rho)\ri)f(\rho)
  \end{split}
\end{equation}
\begin{equation}\label{estimates4}
 \begin{split}
   W_{an}&=\frac{\p^2 W}{\p x^a\p x^n}-\Gamma_{an}^bW_b\\
      & =  \lf(1+G(\rho)\ri)f(\rho)
  \end{split}
  \end{equation}
Hence
\begin{equation}\label{estimates5}
\begin{split}
  I &= \sum_{1\le
   a,b\le n} \frac{W_{ab}^2}{|DW|^2} +(1+G(\rho))f(\rho)\sum_{1\le
   a,b\le n} \frac{W_{ab}^2}{|DW|^2}
\end{split}
\end{equation}
\begin{equation}
\begin{split}
  II &=\sum_{1\le
   a,b,d\le n}\frac{W^bW^d}{|DW|^2}\lf(\frac{W_{ab}}{|DW|}\ri)
   \lf(\frac{W_{ad}}{|DW|}\ri)\\
&+\sum_{1\le
   a,b,c,d\le n}(g^{ac}-\delta^{ac})\frac{W^bW^d}{|DW|^2}\lf(\frac{W_{ab}}{|DW|}\ri)
   \lf(\frac{W_{cd}}{|DW|}\ri)\\
   &=\frac{1}{|DW|^4}\lf(\sum_{1\le i,j,k\le n-1}W^iW^jW_{ki}W_{kj}\ri)+(1+G(\rho))f(\rho)\sum_{1\le
   a,b\le n} \frac{W_{ab}^2+|W_{ab}|}{|DW|^2}\\
   &=\frac{1}{|DW|^4}\sum_{1\le k\le n-1}\lf(\sum_{1\le i\le n-1}W^iW_{ki}\ri)^2+(1+G(\rho))f(\rho)\sum_{1\le
   a,b\le n} \frac{W_{ab}^2+|W_{ab}|}{|DW|^2}
\end{split}
\end{equation}
Hence
\begin{equation}\label{estimates6}
\begin{split}
II &\le
    \frac{1}{|DW|^4}\sum_{1\le i\le n-1}(W^i)^2\sum_{1\le k, i\le n-1}(W_{ki})^2+(1+G(\rho))f(\rho)\sum_{1\le
   a,b\le n} \frac{W_{ab}^2+|W_{ab}|}{|DW|^2}\\
   &\le \frac{2}{|DW|^4}\sum_{1\le i\le n-1}(W^i)^2\sum_{1\le k, i\le n-1}
    \lf[w_{ki}^2+(1+G(\rho))^2f^2(\rho)\ri]\\
    &+(1+G(\rho))f(\rho)\sum_{1\le
   a,b\le n} \frac{W_{ab}^2+|W_{ab}|}{|DW|^2}\\
   &\le \frac{4}{|DW|^4}\sum_{1\le i\le n-1} \lf[w_i^2+(1+G(\rho))^2f^2(\rho)
   \ri]\sum_{1\le k, i\le
   n-1}w_{ki}^2+\frac{(1+G(\rho))^2f^2(\rho)}{|DW|^2}\\
   &+(1+G(\rho))f(\rho)\sum_{1\le
   a,b\le n} \frac{W_{ab}^2}{|DW|^2}\\
   &\le
   \frac{8n\lf[G^2(\rho)(1+f^2(\rho))+f^2(\rho)\ri]}{|DW|^4}\sum_{1\le k, i\le
   n-1}w_{ki}^2+\frac{(1+G(\rho))^2f^2(\rho)}{|DW|^2}\\
   &+(1+G(\rho))f(\rho)\sum_{1\le
   a,b\le n} \frac{W_{ab}^2}{|DW|^2}
  \end{split}
\end{equation}
So
\begin{equation}\label{estimates7}
\begin{split}
|A|^2&\ge\lf[1-(1+G(\rho))f(\rho)\ri]\frac{\sum_{1\le a,b\le
n}W_{ab}^2}{|DW|^2}\\
&-\frac{16n\lf[G^2(\rho)(1+f^2(\rho))+f^2(\rho)\ri]}{|DW|^4}\sum_{1\le
k, i\le
   n-1}w_{ki}^2-\frac{(1+G(\rho))^2f^2(\rho)}{|DW|^2}
\end{split}
\end{equation}

 Hence there exist $\gamma>0$ and
$r>\rho_1>0$ depending only on the bounds in ({\bf a1})--({\bf
a3}) and $n$ such that if $\rho\le \rho_1$ and $G(\rho)\le
\alpha$, then
$$
\lf|(1+G(\rho))f(\rho)\ri|\le \frac14,
$$
and

$$ \lf|16n\lf[G^2(\rho)(1+f^2(\rho))+f^2(\rho)\ri]\ri|\le\frac14|DW|^2.
$$
And so
\begin{equation}\label{estimates8}
\sup_{|x'|\le \rho}\sum_{ij}w_{ij}^2\le C(1+G^2(\rho)).
\end{equation}
Since $w_i=0$ at $x'=0$, we have
\begin{equation}\label{estimates9}
G(\rho)\le  C\rho ( 1+ G(\rho))
\end{equation}
provided $\rho\le \rho_1$ and $G(\rho)\le \gamma$. Hence

\begin{equation}\label{estimates9-1}
G(\rho)\le \frac{C\rho}{1-C\rho}
\end{equation}
provided $\rho\le \rho_1$, $G(\rho)\le \gamma$  and
$C\rho_1<1$. Now choose $\rho_0$ such that
$0<\rho_0<\rho_1$, $C\rho_0<1$ and
$\frac{C\rho}{1-C\rho}\le \frac\alpha2$.

Let $\rho*\le \rho_0$ be the supremum of $\rho$ such that
$w$ can be extended on $|x'|\le\rho$ so that $(x',w(x'))$
is part of $N$ and such that $G(\rho)\le \frac\alpha2$. We
claim that $\rho^*=\rho_0$. Suppose $\rho^*<\rho_1$. Since
$|\p w|\le \frac\alpha2$ in $|x'|<\alpha^*$, $w$ can be
extended to $|x'|=\rho^*$ and beyond. That is, we can find
$\rho^*<\rho_1\le \rho_0$ such that $w$ can be extended to
$|x'|\le \rho_1<\rho_0$ such that $G(\rho_1)\le \alpha$. We
have
$$
G(\rho_1)\le \frac{C\rho_1}{1-C\rho_1}\le \frac\alpha2.
$$
This contradicts the definition of $\rho^*$.
 \end{proof}

\bibliographystyle{amsplain}

\end{document}